\documentclass[singlespace]{article}
\usepackage{inputenc}
\usepackage{cite}
\usepackage{hyperref}
\usepackage{url}
\usepackage{amsmath}
\usepackage{mathtools}
\usepackage{bm}
\usepackage{amsfonts}
\usepackage{algorithm}
\usepackage{algorithmic}
\usepackage{amssymb}
\usepackage{amsthm,paralist}
\usepackage{graphicx}
\usepackage{color}
\usepackage{subfigure}   
\usepackage{diagbox} 
\usepackage{stfloats}  
\usepackage{footnote}
 \usepackage{multirow}  
 \usepackage{booktabs} 
\usepackage[scriptsize,bf]{caption}
\usepackage{amssymb}
\usepackage{diagbox} 
\usepackage{slashbox}
\usepackage{stfloats}  
\usepackage{graphics,graphicx}
\usepackage{epstopdf}
\usepackage[capitalise]{cleveref} 
\usepackage{lipsum}

\crefname{equation}{}{} 

\usepackage{inputenc}
\usepackage{amsmath}
\usepackage{mathtools}
\usepackage{amsfonts}
\usepackage{algorithm}
\usepackage{algorithmic}
\usepackage{amssymb}
 \usepackage{multirow}  
 \usepackage{booktabs} 

\newcommand{\rvs}[1]{{\color{black}#1}}

\def\argmin{\mathop{\rm argmin}}

\def\dR{\mathbb{R}}

\def\avg{\mathrm{avg}}
\def\eig{\mathrm{eig}}

\usepackage{array} 
\newcolumntype{C}{>{$\displaystyle}c<{$}}

\newtheorem{rmk}{Remark}[section] 
\newtheorem{conjecture}{Conjecture}[section] 
\newtheorem{theorem}{Theorem}[section]
\newtheorem{lemma}{Lemma}[section]

\newtheorem{proposition}{Proposition}[section]

\title{Understanding Limitation of Two Symmetrized Orders by Worst-case Complexity}

\author{
Peijun Xiao
\thanks{Coordinated Science Laboratory, Department of ISE, University of Illinois at Urbana-Champaign, Urbana, IL 
  (\texttt{peijunx2@illinois.edu}).}
\and 
Zhisheng Xiao
\thanks{Computational and Applied Mathematics, University of Chicago
  (\texttt{zxiao@uchicago.edu}).}
\and 
Ruoyu Sun
\thanks{Coordinated Science Laboratory, Department of ISE, University of Illinois at Urbana-Champaign, Urbana, IL (\texttt{ruoyus@illinois.edu}).}
}

\ifpdf
\hypersetup{
  pdftitle={ Gaussian Back Substitution and Double Sweep in Block Decomposition},
  pdfauthor={P. Xiao, Z. Xiao, and R. Sun}
}
\fi

\begin{document}
\maketitle


\begin{abstract}

\rvs{Update order is one of the major design choices of  block decomposition algorithms.}
\rvs{There are at least two classes
of deterministic update orders: nonsymmetric (e.g. cyclic order) and symmetric (e.g. Gaussian back substitution or symmetric Gauss-Seidel)}.
\rvs{Recently}, Coordinate Descent (CD) with cyclic order was shown to be $O(n^2)$ times slower than
randomized versions in the worst-case. 
A natural question arises: can the symmetrized orders achieve faster convergence rates
than the cyclic order, or even getting close to the randomized versions? 
In this paper, we give a negative answer to this question. We show that both Gaussian \rvs{b}ack \rvs{s}ubstitution (GBS) and symmetric Gauss-Seidel (sGS)
suffer from the same slow convergence issue as the cyclic order in the worst case.  
In particular, we prove that for unconstrained problems, both GBS-CD and sGS-CD can be $\mathcal{O}(n^2)$ times slower than R-CD.
\rvs{Despite unconstrained problems,  we also empirically study} linearly constrained problems with quadratic objective: we empirically demonstrate that the convergence speed of GBS-ADMM and sGS-ADMM can be roughly $\mathcal{O}(n^2)$ times slower than randomly permuted ADMM.

\end{abstract}





\section{Introduction}

Block decomposition is a simple yet powerful idea for solving large-scale computational problems.
This idea is the key component of several popular methods such as
Coordinate Descent, Stochastic Gradient Descent (SGD)
and Alternating Direction Method of Multipliers (ADMM). \rvs{We first review the background of CD.}

\subsection{Background of Coordinate Descent Methods}


\rvs{CD} methods are iterative methods which update some coordinates of the variable vector while fixing the other coordinates at each \rvs{update}. 
CD is  a popular choice for solving large-scale optimization problems (see \cite{wright2015coordinate} for a survey), 
including glmnet package for LASSO \cite{friedman2010regularization},
 libsvm package for  support vector machine (SVM) \cite{platt1999fast,hsieh2008dual, chang2011libsvm}, 
tensor decomposition \cite{kolda2009tensor},  
resource allocation in wireless communications \cite{shi2011iteratively}, to name a few.

One of the major design choices for CD methods is the update rule. 
\rvs{
In CD methods, we partition the coordinates into $n$ blocks, and update these blocks (possibly multiple times) according to certain order at each epoch. 
More formally, for each epoch $k$, we denote 
$\sigma(k)$ as a finite sequence with values chosen from 
$\{1, 2, \dots, n \}$, often with length at least $n$. 
A natural choice of the update rule is the cyclic order (\rvs{a.k.a. Gauss-Seidel order}), where $\sigma(k) = (1, \cdots, n)$ for all $k$,
and the corresponding version of CD is called cyclic CD (C-CD).
}

In general, there are two  classes of update rules: deterministic  rules, and randomized  rules. 
For the deterministic update rules, the most basic variant is
\rvs{C-CD which is discussed above}; another popular one is symmetric Gauss-Seidel CD (a.k.a. double-sweep CD).
\rvs{In symmetric Gauss-Seidel CD, $\sigma(k) = (1, 2, \cdots, n-1, n, n-1, \cdots, 1)$ for each epoch index $k$.} 
\rvs{Inspired by Gaussian back substitution ADMM, we also consider GBS-CD which consists of a prediction step and a correction step, where the correction step is just one update in C-CD. 
See the details of the two algorithms in \cref{Algorithm: n-block sGS-CD} and \cref{Algorithm: n-block GBS-CD}.
Notice that the first half of the update in these two symmetric variants are the same as the update in the cyclic order.}
For randomized variants, a popular choice in academia is 
\rvs{randomized} CD (R-CD), and a more popular method in practice is randomly permuted CD (RP-CD) \rvs{\cite{shalev2013stochastic, yen2015sparse, chang2008coordinate, hsieh2008dual}}. 
\rvs{In R-CD, $\sigma(k)$ is obtained by randomly choosing an index from the set $\{1, \cdots, n\}$ with replacement for $n$ times.
In RP-CD, $\sigma(k)$ is a randomly picked permutation from the set of all permutations of $\{1, \cdots, n\}$.
In both cases, $\sigma(k)$ is chosen independently from the other epochs.} 
\rvs{In terms of convergence, \cite{powell1973search} shows C-CD diverges for certain nonconvex objectives (with $n \geq 3$), while the randomized version is convergent \cite{wright2015coordinate}. }





\subsection{Motivation}
While CD is observed to be much faster than gradient descent methods (GD)
\rvs{(e.g. see figures in \cite{gordon2015coordinate}
and \cite{sun2016worst})},
is there any theoretical evidence for this observation? 
\rvs{It has been shown that the complexity of R-CD 
is $\tau \triangleq
L / L_{\min}$ times better than that of GD, where
$L$ is the Lipschitz constant of the gradient
 and $L_{\min}$ is the minimal coordinate-wise Lipschitz constant
 of the gradient (see, e.g., \cite{leventhal2010randomized,nesterov2012efficiency}).}
For normalized problems (i.e., all coordinate-wise Lipschitz constants are equal), we have $\tau \in [1, n]$, which
\rvs{implies} that R-CD is $1$ to $n$ times faster than GD;
\rvs{see the figures in \cite{sun2016worst} for experiments
that match with this theoretical gap}.  
Since \rvs{the appearance of \cite{leventhal2010randomized,nesterov2012efficiency}}, most researchers have focused on randomized variants of CD \cite{shalev2013stochastic,richtarik2014iteration,lu2015complexity,zheng5873randomized,
    lin2015accelerated,zhang2017stochastic,fercoq2015accelerated,liu2015asynchronous,
    patrascu2015efficient,hsieh2015passcode}. 

An open question raised in  \cite{sun2016worst} is: does there exist a deterministic variant of CD that achieves similar convergence speed to R-CD? This question is interesting due to several reasons. First, the complexity of deterministic
algorithms is theoretically important, partly because we only have access to pseudo-randomness instead of randomness in practice.
Second, it is not always feasible or easy to randomly pick coordinates \rvs{for high dimensional problems due to time constraints or memory constraints \cite{coddington1997random, thomas2009comparison}}.
Third, understanding deterministic CD may help us better understand other algorithms such as ADMM,
as randomized versions for those methods can be difficult to analyze. 
\rvs{We will introduce ADMM and elaborate the third reason in details in Section \ref{sec: ADMM results}.}




\subsection{Main Contributions}\label{subsec: main contributions}

We mainly study the two symmetrized versions of Gauss-Seidel order: symmetric Gauss-Seidel rule and Gaussian Back Substitution rule. 
Both orders can be applied to any block decomposition type methods,
such as CD and ADMM. 
{\rvs We will provide rigorous proofs that 
the worst-case convergence rate of sGS-CD and GBS-CD
 is similar to C-CD,
 thus can be $O(n^2)$ times
 slower than R-CD. 
}
More specifically, the main contributions of this paper are summarized as below.
\begin{itemize}
    \item We prove that \rvs{for certain convex quadratic problems},
    sGS-CD and GBS-CD converge at least
    $\mathcal{O}(n)$ times slower than GD and $\mathcal{O}(n^2)$ times slower than R-CD.
    Upper bounds of these two methods are also proved for quadratic problems,
    indicating that these gaps are tight. 
    Therefore, these two symmetrized update orders are much slower than the randomized order in the worst case. 
  
     
     \item 
     \rvs{ For constrained problems, to illustrate the slow convergence of
   sGS-ADMM and GBS-ADMM, we 
propose a few examples such that in the experiments,}
GBS-ADMM and sGS-ADMM are $\mathcal{O}(n)$ times slower than the single-block method ALM and roughly $\mathcal{O}(n^2)$ times slower than RP-ADMM.

\end{itemize}

\subsection{Notation and Outline}\label{subsec: Outline}

$\quad $

\textbf{Notation}. 
Before we state the algorithms, we introduce the notation used in this paper. Throughout the paper, \rvs{given a matrix $A \in \dR^{n \times n}$,}
\rvs{l}et $\lambda_{\max}(A), \lambda_{\min}(A)$, and $\lambda_{\avg}(A)$ respectively  denote the maximum eigenvalue, minimum \emph{non-zero} eigenvalue, and average eigenvalue of $A$ respectively;
sometimes we omit the argument $A$ and just use $\lambda_{\max}$, $\lambda_{\min}$ and $\lambda_{\avg}$. We denote the set of the eigenvalues of a matrix $A$ as $\eig(A)$ (allowing repeated elements if an eigenvalue has multiplicity larger than $1$).
The condition number of $A$ is defined as  \rvs{ $\kappa(A) = \frac{\lambda_{\max}(A)}{\lambda_{\min}(A)}$, and we might just use $\kappa$ if the matrix is clear in the context}.

The less widely used notations are summarized below. An important notion is 
$\kappa_{\mathrm CD} = \frac{\lambda_{\mathrm avg}(A)}{\lambda_{\min}(A)}$,
which is the key parameter that determines the convergence speed of a block-decomposition method.
Further, we denote $A_{ij}$ as the $(i,j)$-th entry of $A$ and $L_i = A_{ii}$ as the $i$-th diagonal entry of $A$.
Denote $L_{\max} = \max_i L_i$ and $L_{\min} = \min_i L_i$ as the maximum/minimum coordinate-wise Lipschitz constant (i.e. maximum/minimum diagonal entry of $A$),
and $L_{\avg} = (\sum_{i=1}^n L_i)/n$ as the average of the diagonal entries of $A$ (which is also the average of the eigenvalues of $A$). We use $I$ to denote the identity matrix.

\textbf{Outline}.
The rest of the paper is organized as follows. 
\rvs{
In \cref{sec: CD algorithms}, we introduce the algorithms GBS-CD and sGS-CD and present the upper bounds and lower bounds of their convergence rates in \cref{sec: results}. 
We discuss how to symmetrize update matrices in \cref{sec: proof techniques} and present the proof of the lower bounds in \cref{sec: lower bound proofs}. The numerical experiments of GBS-CD and sGS-CD are given in \cref{sec:experiments}. 
In \cref{sec: ADMM results}, we introduce the symmetrization rules of ADMM and our results of ADMM.
In \cref{sec:conclusions}, we summarize the paper and discuss future research directions. Additional proofs of intermediate technical results and experiments are provided in the appendix.
}


\section{Coordinate Descent Results}\label{sec: CD algorithms}

\rvs{When applying CD to unconstrained problems, we focus on solving convex quadratic functions in the rest of the paper.}
\rvs{
Consider solving the following convex quadratic problem
\begin{equation}\label{least square optimization formulation for sGS-CD}
  \min_{x \in \dR^n }  \quad     \frac{1}{2} \|Ax - b\|^2.\\
\end{equation}}
\rvs{where $A \in \dR^{n \times n} $, $Q \triangleq A^TA$, $ Q_{ii} \neq 0, \ \forall i$ and $b \in \mathcal{R}(A)$.
We can assume $b \in \mathcal{R}(A) $, since otherwise the minimum value of the objective function will be $- \infty$. 
We can assume $Q_{ii} \neq 0, \ \forall i$, since when some $Q_{ii} = 0 $, all entries in the $i$-th row and the $i$-th column of $Q$ should be zero, which means that the $i$-th variable does not affect the objective and thus can be deleted.} 

\subsection{Review of the Update Matrix of C-CD Method}
\label{appen: C-CD update matrix}
\rvs{We first review how to derive the update matrix of C-CD as it will be used in the later discussion of GBS-CD and sGS-CD.}
\rvs{
Given the matrix $Q$, we first denote the lower triangular matrix as
\begin{equation}\label{eq: Gamma}
\Gamma \triangleq  \begin{bmatrix}
a_1^T a_1 &    0      &   \ldots  &   \ldots  & 0 \\
a_2^T a_1 & a_2^T a_2 &    0      &   \ldots  & 0 \\
a_3^T a_1 & a_3^T a_2 &   \ddots  &   \ddots  & 0 \\
  \vdots  &  \vdots   &   \ddots  &   \ddots  & \vdots \\
a_n^T a_1 & a_n^T a_2 &   \ldots  & a_n^T a_{n-1} & a_n^T a_{n}
\end{bmatrix}.
\end{equation}

}
\rvs{We illustrate how to derive the update matrix of C-CD by considering a simple case with $n = 3, d_i = 1, \forall i$, and let $a_i = A_i \in \dR^{3 \times 1}$; the case for general $n$ is quite similar and omitted. Denote $x^{k+1}$ to be the iterate at epoch $k+1$, then the update equations at epoch $k$ can be written as}
\rvs{
\begin{equation}\nonumber
\begin{split}
   & a_1^T(a_1 x_1^{k+1} + a_2 x_2^{k} + a_3 x_3^k - b) = 0, \\
   & a_2^T(a_1 x_1^{k+1} + a_2 x_2^{k+1} + a_3 x_3^k - b) = 0, \\
   & a_3^T(a_1 x_1^{k+1} + a_2 x_1^{k+1} + a_3 x_3^{k+1} - b ) = 0.
\end{split}
\end{equation}
The above update equations can be reformulated into
\begin{equation}\label{update matrix of 123}
   \begin{bmatrix}
   a_1^T a_1 &     0     &      0     \\
   a_2^T a_1 & a_2^T a_2 &      0     \\
   a_3^T a_1 & a_3^T a_2 & a_3^T a_3  \\
   \end{bmatrix}
   x^{k+1} =
   \begin{bmatrix}
   0 & -a_1^T a_2 &  -a_1^Ta_3     \\
   0 &     0      &  -a_2^Ta_3     \\
   0 &     0      &      0         \\
   \end{bmatrix}
    x^k + 
    A^T b.
\end{equation}
}
\rvs{
Combining the notation of $\Gamma$ in \cref{eq: Gamma} with $n$ equals to 3, we have  $  \Gamma x^{k+1}  = (  \Gamma - Q  ) x^k + A^T b  $.
This leads to 
$  x^{k+1}  = (I  - \Gamma^{-1} Q ) x^k + \Gamma^{-1}[A^T b].    $
 It is not hard to verify that the optimal solution $x^* = A^{-1}b $
 satisfies $ \Gamma^{-1}[A^T b] = x^* - (I  - \Gamma^{-1} Q )x^*  $,
 thus 
\begin{equation}
    \label{eq: appendix update of GBS-CD}
    x^{k+1} - x^*  = (I  - \Gamma^{-1} Q ) ( x^k - x^*),
\end{equation}
where $I  - \Gamma^{-1} Q$ is the update matrix of C-CD.
} 

\subsection{Two Coordinate Descent Methods}\label{sec: Algorithms}

In this subsection, we formally present sGS-CD and GBS-CD. 
For simplicity, throughout the paper, we only discuss the case that
the block size is $1$.

\subsubsection{Symmetric Gauss-Seidel Order}\label{subsec: DS_alg}
sGS-CD, also called Aitken's double sweep method \cite{floudas2001encyclopedia},
is presented in \cref{Algorithm: n-block sGS-CD}.

In each epoch, sGS-CD performs a forward pass and a backward pass. 
In the forward pass, we update the coordinates in the order $(1, 2, \dots, n)$. \rvs{The forward pass in sGS-CD is the same as one epoch of C-CD}.
In the backward pass, we update the coordinates in the reverse order $(n-1, \dots, 1)$. 
We call an update rule which updates the coordinates in the order of $(1, 2, \dots, n, n-1, \dots, 1)$ as the \textbf{symmetric Gauss-Seidel (sGS) update rule}.

\begin{algorithm}
\caption{\small sGS-CD }
\label{Algorithm: n-block sGS-CD}
\begin{algorithmic}[1]
\FOR{$k=0,1,2,\ldots,$}
\STATE  \textit{Forward Pass}:
\STATE $\quad x_{1}^{k+\frac{1}{2}} \in \argmin_{x_{1}} f\left(x_{1}, x_{2}^{k-1}, x_{3}^{k-1}, \ldots x_{n}^{k-1}\right)$
\STATE $\quad \cdots$
\STATE $\quad x_{n}^{k+\frac{1}{2}} \in \argmin_{x_{n}} f\left(x_{1}^{k+\frac{1}{2}}, x_{2}^{k+\frac{1}{2}}, x_{3}^{k+\frac{1}{2}}, \ldots x_{n}\right)$
\STATE \textit{Backward Pass}:
\STATE $\quad x_{n}^{k+1}  = x_{n}^{k+\frac{1}{2}} $
\STATE $\quad x_{n-1}^{k+1}  \in \argmin_{x_{n-1}} f\left(x_{1}^{k+\frac{1}{2}}, \ldots, x_{n-2}^{k+\frac{1}{2}}, x_{n-1}, x_{n}^{k+1}\right)$
\STATE $\quad \cdots$
\STATE $\quad x_{1}^{k+1} \in \argmin_{x_{1}} f\left(x_{1}, x_{2}^{k+1}, x_{3}^{k+1}, \ldots x_{n}^{k+1}\right)$
\ENDFOR
\end{algorithmic}
\end{algorithm}

\begin{rmk} \label{rmk: same update in backward pass} $ \text{}$  
\rvs{The line 7 in \cref{Algorithm: n-block sGS-CD} is the same as 
$$x_{n}^{k+1}  \in \argmin_{x_{n}} f\left(x_{1}^{k+\frac{1}{2}}, x_{2}^{k+\frac{1}{2}}, \ldots, x_{n-1}^{k+\frac{1}{2}}, x_{n}\right),$$
because $x_{n}^{k+1}$ is optimal given current values of the other coordinates. 
}
\end{rmk}

\rvs{
\cref{rmk: same update in backward pass} implies that the backward pass of sGS-CD updates the coordinates in order $(\mathbf{n},n-1,...,1)$. 
Therefore, the update matrix of the backward pass can be written as $I - \Gamma^{-T}Q$.
The update matrix of sGS-CD is simply the product of update matrices of the forward and backward pass:
\begin{align}\label{eq: sGS-CD update matrix for iterate}
    (I - \Gamma^{-T}Q)(I - \Gamma^{-1}Q).
\end{align}
}

\subsubsection{Gaussian Back Substitution Order}\label{subsec:alg_GBS}


Next, we present GBS-CD for solving the quadratic problem 
\cref{least square optimization formulation for sGS-CD}.
\rvs{GBS-CD consists of two steps in each epoch: a prediction step and a correction step.}
The prediction step of GBS-CD is a regular epoch of C-CD method,
which can be written as 
\begin{align}\label{eq: GBS-CD prediction}
    \tilde{x}^{k} - x^\ast = (I - \Gamma^{-1}Q)(x^k - x^\ast),
\end{align}
\rvs{This expression is the same as \cref{eq: appendix update of GBS-CD} 
besides that we replace the notation $x^{k+1}$ by $\tilde{x}^{k}$ to denote the iterate after the prediction step in epoch $k$.}
Let \rvs{$\Gamma^{-T}$ be the transpose of the inverse of $\Gamma$ defined in \cref{eq: Gamma} and 
$D$ be the diagonal matrix of $Q$.
We denote a matrix $B$ as}
\begin{equation}\label{eq: matrix B in GBS-CD}
 B \triangleq  \begin{bmatrix}
1 &   0 \\
0 & \Gamma^{-T}_{2:n}
\end{bmatrix}
\begin{bmatrix}
1 &   0 \\
0 & D_{2:n}
\end{bmatrix}.
\end{equation}
Here $\Gamma^{-T}_{2:n}$ and $D_{2:n}$ are the sub-matrices obtained by excluding the first row and the first column of \rvs{$\Gamma^{-T}$} and \rvs{$D$} respectively. 

The correction step of GBS-CD at \rvs{the $k$}-th epoch is defined as as
\begin{equation}\label{eq:correction step in GBSCD}
    x^{k+1} = x^k - B(x^k - \tilde{x}^k).
\end{equation}

This step \cref{eq:correction step in GBSCD} corrects the prediction of $\tilde{x}^{k}$ in \cref{eq: GBS-CD prediction}.
\rvs{Since $B$ is an upper triangular matrix, \cref{eq:correction step in GBSCD} can be implemented by back substitution, i.e., update $(x_n, \cdots, x_2)$ in a sequential order.} 
\rvs{Combing the prediction step and the correction step together, the update at the $k$-th epoch can be written as}
\begin{align}\label{eq:GBS update rule}
    x^{k+1} - x^\ast = (I - B \Gamma^{-1}Q) (x^{k} - x^\ast).
\end{align}
This means the update matrix of GBS-CD is
\begin{align}\label{eq:GBS update matrix iterate}
    I - B \Gamma^{-1}Q.
\end{align}
\rvs{We formally define GBS-CD in \cref{Algorithm: n-block GBS-CD}.}


\begin{rmk}
\textbf{Terminology: ``iteration", ``pass'' and ``epoch''}.
Throughout the paper, we use ``iteration'' to denote one step
of updating one coordinate (or one block) 
in an algorithm. 
\rvs{Suppose we have $n$ coordinates, then in C-CD, 
one iteration means one update of one coordinate, and
one epoch means one update of the coordinates in the order of $\{1, \cdots, n\}$}.
\rvs{For symmetrized versions of Gauss-Seidel order, one ``pass'' consists of the update of all the coordinates regardless of the order (one pass can
be a ``forward pass'' or a ``backward pass''),
and one ``epoch'' consists of two passes.
}
\rvs{When solving quadratic problems}, each epoch of GD (or C-CD) takes $\mathcal{O}(n^2)$ number of operations, and each epoch of sGS-CD (or GBS-CD) takes twice the number of operations of GD (or C-CD) \rvs{as it has two passes per epoch}.
Since each epoch of different algorithms takes different numbers of operations, for a fair comparison, we focus on \rvs{comparing} the total complexity of the algorithms, \rvs{i.e. the product of the number of operations per epoch and the total number of epochs to reach some stopping criteria (e.g. the gradient norm is less than certain threshold)}.
\end{rmk}


\begin{algorithm}
\caption{\normalsize GBS-CD }
\label{Algorithm: n-block GBS-CD}
\begin{algorithmic}
\FOR{$k=0,1,2,\ldots,$}
\STATE  \textit{Prediction Step}:
\STATE $\quad \tilde{x}_{1} \in \argmin_{x_{1}} f\left(x_{1}, x_{2}^{k-1}, x_{3}^{k-1}, \ldots x_{n}^{k-1}\right)$
\STATE $\quad \cdots$
\STATE $\quad \tilde{x}_{n} \in \argmin_{x_{n}} f\left(\tilde{x}_{1}, \tilde{x}_{2}, \tilde{x}_{3}, \ldots x_{n}\right)$
\STATE \textit{Correction Step}:
\STATE $\quad x^{k+1} = x^k - B(x^k - \tilde{x}^k)$, where $B$ is defined
in \cref{eq: matrix B in GBS-CD}
\ENDFOR
\end{algorithmic}
\end{algorithm}


                   

\section{Main Results: Upper Bounds and Lower Bounds}\label{sec: results}



We first summarize the main results in \cref{full table of total complexity}. The upper bounds will be given in \cref{prop: upper bound sGS-CD}, \cref{prop: upper bound GBS-CD objective},  and the lower bounds will be given in \cref{thm: sGS-CD lower bound} and \cref{thm: GBS lower bound}. In this table, we ignore the  $\log 1/\epsilon$ factor, which is  necessary for an iterative algorithm to achieve error $\epsilon$. 

\begin{table}[!htbp]
    \centering
    \caption{Complexity of sGS-CD, GBS-CD, C-CD, R-CD, GD for \rvs{quadratic problems in} equal-diagonal cases (ignoring a $\log 1/\epsilon$ factor)} \label{full table of total complexity}
    \renewcommand\arraystretch{1.3}
    \begin{tabular}{|c|c|c|c|c|}
        \hline
        Algorithms  &   $\kappa$ & $\kappa_{\mathrm CD}$    \\
        \hline
        sGS-CD Upper bound (\cref{prop: upper bound sGS-CD}) & $n^3 \kappa $ & $n^4 \kappa_{\mathrm CD} $ \\ 
        sGS-CD Lower bound (\cref{thm: sGS-CD lower bound}) & $\frac{1}{40} n^3 \kappa$  & $\frac{1}{40} n^4 \kappa_{\mathrm CD}$ \\
        \hline
        GBS-CD Upper bound (\cref{prop: upper bound GBS-CD objective}) & $n^3 \kappa $ & $n^4 \kappa_{\mathrm CD} $ \\ 
        GBS-CD Lower bound (\cref{thm: GBS lower bound}) & $\frac{1}{15} n^3 \kappa$  & $\frac{1}{15} n^4 \kappa_{\mathrm CD}$ \\
        \hline
        GD & $n^2 \kappa$ & -- \\
        \hline
        R-CD & -- & $n^2 \kappa_{\mathrm CD}$ \\
        \hline
        C-CD Upper Bound & $n^3 \kappa$ & $n^4 \kappa_{\mathrm CD}$ \\
        C-CD Lower Bound & $\frac{1}{40} n^3 \kappa$ & $\frac{1}{40} n^4 \kappa_{\mathrm CD}$ \\
        \hline
    \end{tabular} 
\end{table}

This table shows that the lower bounds match the upper bounds up to some constant
factor. In addition, the table reveals the relations between the worst-case complexity of sGS-CD, GBS-CD, GD, R-CD and C-CD.

The main implications of our results are the following:
\begin{itemize}
 \item  In the worst case, sGS-CD and GBS-CD are $\mathcal{O}(n)$ times slower than GD, and $\mathcal{O}(n^2)$ times slower than R-CD. 
 \item 
 sGS-CD and GBS-CD are as slow as C-CD up to a constant factor in the worst case.
\end{itemize}


Now we formally state the upper bounds and lower bounds on the convergence rate of sGS-CD and GBS-CD. We let $f^*$ denote the minimum value of a function $f$.

\begin{proposition}\label{prop: upper bound sGS-CD}(Upper bound of sGS-CD)
\rvs{Consider solving the problem \cref{least square optimization formulation for sGS-CD}},
\rvs{and for} any $x^0 \in \dR^n $, let $x^k$ denotes the output of sGS-CD after $k$ epochs, then
\begin{align}\label{eq: SGSCD upper bound prop}
     f(x^{k+1}) - f^* \leq \left(\min \left\{ 1 - \frac{1}{ n \kappa  } \frac{L_{\min}}{L_{\avg}} ,  1 - \frac{L_{\min} }{ L (2 + \log n/ \pi)^2 } \frac{1}{ \kappa  } \right\}\right)^2  (f(x^k) - f^*).
\end{align}


\end{proposition}

\begin{proposition}
\label{prop: upper bound GBS-CD objective}(Upper bound of GBS-CD)
\rvs{Consider solving the problem \cref{least square optimization formulation for sGS-CD}},
   and for any $x^0 \in \dR^n $, let $x^k$ denotes the output of GBS-CD after $k$ epochs, then
\begin{align}\label{eq: GBS upper bound prop objective}
     f(x^{k+1}) - f^* \leq 
     \left(1 - \frac{1}{\kappa \cdot \min \left\{   \sum_i L_i, (2 + \frac{1}{\pi} \log n)^2  L    \right\}}
     \right)^{2}
     (f(x^k) - f^*)
\end{align}


\end{proposition}

\begin{theorem}\label{thm: sGS-CD lower bound}(Lower bound of sGS-CD)
	For any initial point $x^0 \in \dR^n $, any $\delta \in (0,1]$, there exists a quadratic function $f(x) = x^TA x - 2b^T x $ such that 
		\begin{align}
		f(x^k) - f^* \geq (1 - \delta) \left( 1 - \frac{ 4\pi^2 }{ n \rvs{ \kappa(A)} }  \right)^{2k + 2 } ( f(x^0) - f^*) , \;  \forall k,
		\end{align}
	where  $x^k$ denotes the output of sGS-CD after $k$ epochs.
 \end{theorem}
 
\begin{theorem}\label{thm: GBS lower bound}(Lower bound of GBS-CD)
    For any initial point $x^0 \in \dR^n $, 
    for any $\delta \in (0,1]$,
    there exists a quadratic function $f(x) = x^TA x - 2b^T x $ such that
    \begin{equation}
        f(x^k) - f^* \geq (1 - \delta) \left( 1 - \frac{ 3\pi^2}{ (12-\pi^2) c n \rvs{\kappa(A)} }  \right)^{2k + 2 } ( f(x^0) - f^*) , \;  \forall k.
	\end{equation}
	where \rvs{$x^k$ denotes the output of GBS-CD after $k$ epochs} and $c \in (0, 1)$ is a constant defined for the quadratic function.
\end{theorem}

The proofs of \cref{thm: sGS-CD lower bound} and \cref{thm: GBS lower bound} will be given in \cref{subsec: proof sGS lower bound} and \cref{subsec: proof GBS lower bound} respectively.
\rvs{We defer the formal proofs of \cref{prop: upper bound sGS-CD} and \cref{prop: upper bound GBS-CD objective} to \cref{subsec: prop: upper bound sGS-CD} and \cref{subsec: prop: upper bound GBS-CD objective}, since they are relatively easy (based on earlier result of \cite{sun2016worst}).}

Now we discuss how to obtain \cref{full table of total complexity} from the above results. We say an algorithm has complexity $\tilde{\mathcal{O}}( g(n, \theta) )$,
if it takes $\mathcal{O}( g(n, \theta)  \log(1/\epsilon) )$ unit operations 
to achieve relative error $\epsilon$.
As we discussed before, each epoch of GD, C-CD and R-CD will take $\mathcal{O}(n^2)$ operations, and each epoch of sGS-CD and GBS-CD will take twice the number of operations of GD. 
Using the fact $ -\ln(1 - z) \geq - z, z \in (0,1) $, one can immediately show that to achieve
$ (1 - 1/u )^k \leq \epsilon $, one only needs $ k \geq u \log(1/\epsilon) $ epoch. 
Thus we can transform the convergence rate to the number of epoch, then the total time complexity. 

Consider the equal-diagonal case for now, then $\frac{L}{ L_{\min} } = \frac{\lambda_{\max}}{\lambda_{\avg}}$ and $\kappa_{\mathrm{CD}} = \frac{L_{\avg}}{\lambda_{\min}} = \frac{\lambda_{\avg}}{\lambda_{\min}}$
and the upper bound of sGS-CD on the convergence rate \cref{eq: SGSCD upper bound prop} can be transformed to the following upper bound of complexity
\begin{equation}
 \tilde{\mathcal{O}}\left( n^3 \kappa \right) \quad \textrm{or } \quad \tilde{\mathcal{O}}\left(n^4 \kappa_{\mathrm{CD}} \right).
\end{equation}
These two quantities are those entries of sGS-CD upper bound in \cref{full table of total complexity}. Similarly, the other bounds on the convergence rates in \cref{thm: sGS-CD lower bound}, \cref{prop: upper bound GBS-CD objective} and \cref{thm: GBS lower bound}  can be transformed to corresponding bounds on the complexity, and they form the rest of \cref{full table of total complexity}.

\section{Symmetrized Update Matrix}\label{sec: proof techniques}


\rvs{ As iterative algorithms can often be written as matrix recursions, we would like to analyze the update matrices of the matrix recursions to study the convergence of the algorithms. In particular, we want to analyze the spectral radius of the update matrix $M$,
which is $\rho(M)$, and compute the lower bound and upper bound of $\rho(M)$. 

\textbf{Difference of upper bound and lower bound.}
For analyzing the upper bounds of convergence rate, researchers often
relax the spectral radius to the spectral norm, i.e., 
use the bound $ \rho(M) \leq \| M\|, \forall M $,
and then obtain an upper bound of $\| M\|$.
In fact, if $ \|  M \| \leq r $, then the matrix recursion
$ x_{k+1} = M x_k $ satisfies  $ \| x_k\| \leq r^k \| x_0\|$. 
 The benefit of analyzing spectral norm is that it has many nice properties, e.g., $ \| M_1 + M_2 \| \leq \| M_1\| + \|  M_2 \| $
 and $ \| M_1 \| \| M_2 \| \leq \| M_1\| \|  M_2 \| .$
 These properties will often make the analysis 
 much easier. 
 
However, to obtain the lower bound of the convergence rates, 
analyzing the spectral norm is not enough in general.
Even if we can prove $ \|  M \| \geq r $, we cannot claim the matrix recursion $ x_{k+1} = M x_k $ satisfies  $ \| x_k\| \geq r^k \| x_0\|$.
It seems that we have to analyze the spectral
radius directly, i.e., proving $\rho(M) \geq r$. 
This is one challenge for analyzing algorithms
with non-symmetric update matrices such as C-CD (e.g.,
\cite{sun2016worst} computes $\rho(M)$ for C-CD).

\textbf{Idea: symmetrization.}
Nevertheless,  sGS and GBS are motivated by 
symmetrizing the Gauss-Seidel rule,
thus we were wondering whether this design principle
could simplify the update matrix in some way.
An important (though simple) observation is that
for any symmetric matrix $M$, $\rho(M) = \| M \|$. 
Thus if we can get a symmetric update matrix,
the analysis of the spectral radius can be reduced
to the analysis of the spectral norm, which shall be more tractable. 
This motivates the first main idea: \textit{we shall transform
the update matrices of sGS and GBS to symmetric matrices}. 
 }
\rvs{

Since we want to analyze the spectral radius, 
then we need to keep the eigenvalues during the ``transformation''.  
Therefore, we shall apply similarity transformation. 
We wish to obtain a matrix $M'$ such that
\begin{itemize}
    \item (R1) $M'$ is symmetric;
    \item (R2) $M' \sim M$, where $M$ is the update matrix
    of sGS or GBS. 
\end{itemize}
Here $\sim$ is the notation of a similar transformation: $M_1 \sim M_2$ means $M_1$ is similar to $M_2$.
The major questions of this section are:
\begin{equation}
\begin{split}
  &  \text{Q1: Does there exist } M'  \text{ s.t. }   M'   \sim M,
    \text{for } M  = M_{\text{sGS-CD}} \text{ and }M_{\text{GBS-CD}}? \\
   &   \text{Q2: If yes, what are they?} 
    \end{split}
\end{equation}
It was not clear whether such a matrix $M$ exits or not.
The original update matrices of them \cref{eq: sGS-CD update matrix for iterate} and \cref{eq:GBS update matrix iterate}
are not symmetric. 
Our intuition is that both sGS and GBS perform
``symmetrization'' which may lead to symmetric update matrices,
but this intuition does not directly lead to a proof.

Fortunately, the answer to Q1 is yes. 
We will present the transformed matrices in the next 
two propositions. 
}

\rvs{
}

\rvs{
\textbf{Two propositions on symmetrized update matrices.}
\begin{proposition}\label{prop: sGS-CD symmetrization}
(Transformation of sGS-CD update matrix)
Consider the matrix $M_{\text{sGS-CD}}$ defined in \cref{eq: sGS-CD update matrix for iterate}, which is the update matrix of sGS-CD. We have
\begin{equation}
    \eig \left( M_{\text{sGS-CD}} \right) =
\eig \left( (I -  A \Gamma^{-T} A^T  ) ( I -  A \Gamma^{-1} A^T ) \right).
\end{equation}
\end{proposition}

\begin{proposition}\label{prop: GBS-CD symmetrization}
(Transformation of GBS-CD update matrix)
Consider the matrix $M_{\text{GBS-CD}}$ defined in \cref{eq:GBS update matrix iterate}, which is the update matrix of GBS-CD. We have
\begin{equation}
    \eig \left( M_{\text{GBS-CD}}  \right) =
\eig \left( I - \Gamma ^{-1} Q \Gamma^{-T} \right).
\end{equation}
\end{proposition}
}

\rvs{
\cref{prop: sGS-CD symmetrization} shows that the update matrix of sGS-CD 
has the same eigenvalues as the matrix $M'_{\text{sGS-CD}} = (I -  A \Gamma^{-T} A^T  ) ( I -  A \Gamma^{-1} A^T )$. 
We denote $J$ as $I -  A \Gamma^{-T} A^T$ and observe that $M'_{\text{sGS-CD}} $ has two nice properties:
\begin{itemize}
    \item $M'_{\text{sGS-CD}} $ is a product of $J$ and $J^T$;
    \item $M'_{\text{sGS-CD}}$ itself is symmetric.
\end{itemize}

These properties implies the following important relation
\begin{align}\label{eq: sGS simpify proof}
    \rho(M_{\text{sGS-CD}}) = \rho(M'_{\text{sGS-CD}}) = \|M'_{\text{sGS-CD}}\| = \|J J^T \| = \|J\|^2,
\end{align}
which would simplify the analysis of sGS-CD, 
as we can utilize the analysis of $J$ (note that $J$ is the update matrix
of C-CD analyzed in \cite{sun2016worst}). 

Similarly, \cref{prop: GBS-CD symmetrization} shows that the update matrix of GBS-CD denoted as $M_{\text{GBS-CD}} $ has the same eigenvalues as a symmetric matrix $M'_{\text{GBS-CD}} = I - \Gamma ^{-1} Q \Gamma^{-T}$. This implies that 
\begin{align}\label{eq: GBS simpify proof}
    \rho(M_{\text{GBS-CD}}) = \rho(M'_{\text{GBS-CD}}) = \|M'_{\text{GBS-CD}}\|.
\end{align}
The analysis of $\|M'_{\text{GBS-CD}}\|$ is rather nontrivial;
see the formal proof later. 
}

\rvs{
In \cref{subsec: proof sGS lower bound} and \cref{subsec: proof GBS lower bound}, we will show how to use \cref{prop: sGS-CD symmetrization} and \cref{prop: GBS-CD symmetrization} to derive the lower bounds of the convergence rates of sGS-CD and GBS-CD. In the next subsection, we will provide the proofs of the 
above propositions.
}


\subsection{Proof of the Propositions}
\rvs{

$\text{ }$

\textbf{Proof of \cref{prop: sGS-CD symmetrization}}:
By the definition of $Q$ when formulating the problem \cref{least square optimization formulation for sGS-CD}, we can decompose $Q$ as $A^TA$ and get
 \begin{align}
 \label{sgs cd update}
 M_{\text{sGS-CD}} =  (I - \Gamma^{-T} Q ) (I - \Gamma^{-1} Q) = (I -  \Gamma^{-T} A^T A ) (I - \Gamma^{-1} A^T A ).
 \end{align}

Since $M_{\text{sGS-CD}}$ is not a symmetric matrix, then we want to apply similarity transformation to $M_{\text{sGS-CD}}$ and get a symmetric matrix.
 \begin{equation}\label{sGS symmetric form in primal form s5}
 \begin{split}
  M'_{\text{sGS-CD}} & =  A M_{\text{sGS-CD}} A^{-1}   \\
  &  = (A -  A \Gamma^{-T} A^T A ) (A^{-1} - \Gamma^{-1} A^T )  \\
  & = (I -  A \Gamma^{-T} A^T  ) A A^{-1} ( I -  A \Gamma^{-1} A^T )   \\
  & = (I -  A \Gamma^{-T} A^T  ) ( I -  A \Gamma^{-1} A^T ).  
 \end{split}
 \end{equation}
 
As $M'_{\text{sGS-CD}}$ is similar to $M_{\text{sGS-CD}}$, they have the same eigenvalues.  $\Box$


\vspace{0.3cm}

\textbf{Proof of \cref{prop: GBS-CD symmetrization}}:
We have
\begin{equation}\label{eq: GBS similar relations 2}
    B\Gamma^{-1} Q  \sim \Gamma^{-1} Q B.
\end{equation}

\begin{lemma}\label{lm: equal eigenvalues}
\rvs{Consider an equal-diagonal matrix $Q$ with $Q_{ii} = 1$ and the corresponding lower triangular matrix $\Gamma$ and matrix $B$ as defined in \cref{eq: matrix B in GBS-CD}.} 
\rvs{Then the eigenvalues of $B^{-1}Q^{-1} \Gamma$ and $\Gamma^T Q^{-1} \Gamma$ are the same.}
\end{lemma}

The key step in the proof of \cref{lm: equal eigenvalues} is to rewrite $B^{-1}Q^{-1} \Gamma$ and $\Gamma^{T} Q^{-1} \Gamma$ into four block matrices and compare the eigenvalues of the block matrices. If the eigenvalues are the same for the same block matrices of $B^{-1}Q^{-1} \Gamma$ and $\Gamma^{T} Q^{-1} \Gamma$, then the eigenvalues of the original matrices are the same by the connections between eigenvalues and trace of matrices. 
See the detailed proof in Appendix \ref{lm with proof: equal eigenvalues}.

By \cref{lm: equal eigenvalues}, we have $\eig (\Gamma^{-1} Q B) = \eig (\Gamma ^{-1} Q \Gamma^{-T})$.  Combining this results with 
\cref{eq: GBS similar relations 2}, we conclude that
$\eig(I - B \Gamma ^{-1} Q) = \eig(I - \Gamma ^{-1} Q \Gamma^{-T})$.
$\Box$
}

\section{Proof of the Lower Bounds}\label{sec: lower bound proofs}

\rvs{In this section, we first present the example we will use for proving the lower bounds.
We then present the formal proofs of the lower bounds.
}

\subsection{Worst-case example}\label{subsec: worst case}

In the proof of \cref{thm: sGS-CD lower bound} and \cref{thm: GBS lower bound}, 
we will consider the following quadratic problem \cref{eq:Ac problem} 
\begin{equation}\label{eq:Ac problem}
\begin{split}
  \min_{x \in \dR^n }  \quad   & x^T A^TA x\\
\end{split}
\end{equation}
with a special case of matrix $A \in \dR^{n \times n}$ such that

\begin{equation}\label{eq:Ac def}
 Q \triangleq A^TA = \begin{bmatrix}
    1 &  c &  \dots  &  c   \\
    c &  1 &  \dots  &  c   \\
    \vdots &  \vdots  &   \ddots &  \vdots \\
    c &  c &   \dots &   1
 \end{bmatrix}
\end{equation}
for a given constant $c \in (0, 1)$.

$Q$ is a positive definite matrix, with one eigenvalue $1-c$ with multiplicity $n-1$ and one
eigenvalue $1-c + cn $ with multiplicity $1$. The condition number of $Q$ is
\begin{equation}\label{kappa expresssion}
  \kappa = \frac{1 - c + cn}{ 1 - c},
\end{equation}
where $\kappa$ approaches infinity when $c$ approaches $1$.

\begin{rmk}
This instance has been analyzed in \cite{sun2016worst} for C-CD and  \cite{lee2018random} for RP-CD. 
Using this example (for $c \rightarrow 1$), \cite{sun2016worst} shows C-CD can be $\mathcal{O}(n^2)$ times slower than R-CD.
Note that when analyzing sGS-CD, we also require $c \rightarrow 1$;
but when analyzing GBS-CD, we allow any $c \in (0, 1)$).
\end{rmk}

\begin{rmk}
(\textbf{Permutation Invariant}). The matrix $Q$ is a permutation-invariant matrix, i.e. $P^TQP = Q$ for any permutation matrix $P$.
This implies that even if one randomly permutes the coordinates at the beginning
and then apply sGS-CD or GBS-CD, the iterates do not change. Thus our lower bounds
would still hold. 
\end{rmk}


\subsection{Proof of the lower bound of sGS-CD}\label{subsec: proof sGS lower bound}

\begin{proof}
\rvs{As discussed in \cref{sec: proof techniques}, based on \cref{eq: sGS simpify proof},
it is sufficient to analyze the lower bound of $\|I -  A \Gamma^{-1} A^T \|$.}

Since we use exactly the same worst case example as in \cite{sun2016worst}, we use some of the results there. \cite[Theorem 3.1]{sun2016worst} can be interpreted as: 
\begin{theorem}\label{thm: CCD Lower bounded}(Theorem 3.1 in \cite{sun2016worst}, re-interpreted)

When solving problem \eqref{eq:Ac problem} where $A$ satisfies \eqref{eq:Ac def} using C-CD, for any initial point $x^0 \in \dR^n $, any $\delta \in (0,1]$, we have
    \begin{align}\label{eq: rewrite Sun Theorem 3_1 eq1}
        \frac{\|\hat{r}^k\|^2}{\|\hat{r}^0\|^2} \geq (1 - \delta)\left( 1 - \frac{ 2\pi^2 }{ n \kappa }  \right)^{2k + 2 },
    \end{align}
    where $\hat{r}^k = A \hat{x}^k$ and $\hat{x}^k$ is the iterate in C-CD at epoch $k$.
\end{theorem}

We claim that \cref{eq: rewrite Sun Theorem 3_1 eq1} implies
\begin{align}
    \|I -  A \Gamma^{-1} A^T  \| \geq (1 - \delta)\left( 1 - \frac{ 2\pi^2 }{ n \kappa }  \right). \label{bound Phat}
\end{align}
This is because
\begin{align*}
    \hat{r}^{k+1} &= A\hat{x}^{k+1} = A(I - \Gamma^{-1} Q)\hat{x}^t\\ &=A\hat{x}^{k} - A\Gamma^{-1}A^T A\hat{x}^k\\
    &=(I-A\Gamma^{-1}A^T)\hat{r}^k,
\end{align*}
and if \cref{bound Phat} does not hold, then
\[
\|\hat{r}^k\| \leq \|I-A\Gamma^{-1}A^T\|^{k+1}\|\hat{r}^0\|<(1 - \delta)\left( 1 - \frac{ 2\pi^2 }{ n \kappa }  \right)^{k + 1 }\|\hat{r}^0\|,
\]
which contradicts \cref{eq: rewrite Sun Theorem 3_1 eq1}.

By  \eqref{bound Phat}, we have 
\begin{align*}
    \| I -  A \Gamma^{-1} A^T \|^2 \geq \left( (1 - \delta) \left( 1 - \frac{ 2\pi^2 }{ n \kappa }  \right)\right)^2,
\end{align*}
and hence we have
    \begin{align}
        \frac{\|x^{k+1}\|^2}{\|x^k\|^2} = \rho(M_{\text{sGS-CD}})^2 = \|I -  A \Gamma^{-1} A^T\|^4 \geq \left( (1 - \delta) \left( 1 - \frac{ 2\pi^2 }{ n \kappa }  \right)\right)^4. \notag
    \end{align}
Then, we can obtain the desired lower bound on iterates of sGS-CD: 
\begin{align}
         \frac{\|x^k\|^2}{\|x_0\|^2} 
        \geq  (1 - \delta)^2\left( 1 - \frac{ 2\pi^2 }{ n \kappa }  \right)^{4k + 4 } 
         \geq  (1 - 2\delta)\left( 1 - \frac{ 4\pi^2 }{ n \kappa }  \right)^{2k + 2}. \notag
    \end{align}
  This also implies the same convergence rate of objective error $r^k = Ax^k$:
  \begin{align}
         \frac{\|r^k\|^2}{\|r_0\|^2} 
         \geq  (1 - 2\delta)\left( 1 - \frac{ 4\pi^2 }{ n \kappa }  \right)^{2k + 2}. \notag
    \end{align}
  
\end{proof}

\begin{rmk}
The above proof is based on symmetrizing the update matrix. In \cref{sec: sGS projection proof}, we will provide an alternative proof that is based on expressing sGS-CD as alternating projections.
\end{rmk}

\subsection{Proof of the lower bound of GBS-CD}\label{subsec: proof GBS lower bound}

\begin{proof}

\rvs{As discussed in \cref{sec: proof techniques}, based on \cref{eq: GBS simpify proof},
it is sufficient to analyze the lower bound of $\|I-   \Gamma^T Q^{-1} \Gamma \|$.}
The eigenvalues of $\Gamma^T Q^{-1} \Gamma$ do not have a clear form, so we decompose the matrix and analyze the spectral norm of the decomposition in \cref{lm: decompose Q inverse}. It can be proved simply by inspection (proof is provided in \cref{sec: sup proofs for GBS lower bound}).

\begin{proposition}\label{lm: decompose Q inverse}
\rvs{For $Q$ defined in \cref{eq:Ac def}, $\tilde{c} := 1-c$ and $J := ee^T$,}
    \vspace{-0.6em}
    \begin{align}\label{eq: decompose Q inverse}
        Q^{-1} = \frac{1}{\tilde{c}} I - \frac{c}{\tilde{c} (\tilde{c}+ cn)} J.
    \end{align}
\end{proposition}
Denote $a = \frac{1}{\tilde{c}}$ and $b = \frac{c}{\tilde{c} (\tilde{c}+ cn)}$, we have that 
\begin{subequations}\label{eq:decompose LQL}
\begin{align}
     & \|\Gamma^T Q^{-1} \Gamma\| \notag \\
    = & \|\Gamma^T(aI - bJ)\Gamma\| \notag\\
    = & \|a \Gamma^T\Gamma - b\Gamma^TJ\Gamma\| \notag\\ 
    \geq & \left\lvert \|a \Gamma^T\Gamma\| - \|b\Gamma^TJ\Gamma\| \right\lvert \label{eq: decompose LQL step 1} \\
    = & \left\lvert \|a \Gamma\Gamma^T\| - \|b\Gamma^TJ\Gamma\| \right\lvert \notag\\
    = & \left\lvert a\| \Gamma\Gamma^T - Q + Q\| - \|b\Gamma^TJ\Gamma\| \right\lvert \notag\\
    \geq & \left\lvert a \left\lvert \|\Gamma\Gamma^T - Q\| - \|Q\|\right\lvert - b \|\Gamma^TJ\Gamma\| \right\lvert \label{eq: decompose LQL step 2},
\end{align}
\end{subequations}
where both the inequalities \cref{eq: decompose LQL step 1} and \cref{eq: decompose LQL step 2} use the reverse triangle inequality, i.e.,  
\begin{subequations}
\begin{align}
    \|a \Gamma^T\Gamma - b\Gamma^TJ\Gamma\| 
    \geq & \left\lvert \|a \Gamma^T\Gamma\| - \|b\Gamma^TJ\Gamma\| \right\lvert, \\
    \| \Gamma\Gamma^T - Q + Q\| \geq & \left\lvert \|\Gamma\Gamma^T - Q\| - \|Q\|\right\lvert.
\end{align}
\end{subequations}
We introduce the following propositions to compute the spectral norm of each term in \cref{eq: decompose LQL step 2}. Proofs of these propositions are provided in the \cref{sec: sup proofs for GBS lower bound}.

\begin{proposition}\label{prop: a few props in GBS proof}
    \rvs{For $Q$ defined in \cref{eq:Ac def}, we have}
    \vspace{-0.6em}
    \begin{align}
        \|Q\| & = 1-c + cn; \\
        \quad \|\Gamma^TJ\Gamma\| & = n + cn(n-1) + \left(c^2 n(n-1)(2n-1)\right) / 6; \\
        \|\Gamma\Gamma^T - Q\| & \gtrapprox (c^2  4 n^{2}) / \pi^{2}.
    \end{align}
\end{proposition}



     

We apply the propositions and plug in the definitions of $a$ and $b$ into \cref{eq:decompose LQL}:
\begin{subequations}\label{eq: plug in coefficients}
\begin{align}
\vspace{-1em}
    & \| \Gamma^T Q^{-1} \Gamma \| 
    \notag\\
    \geq & \left| a \left| \|\Gamma\Gamma^T - Q\| - \|Q\|\right| - b \|\Gamma^TJ\Gamma\| \right| 
    \label{eq: two abs}
    \\
    \geq &\left( a \left(  \frac{4 c^2 n^{2}}{\pi^{2}} - (1-c + cn) \right) - b \left(n + cn(n-1) + \frac{c^2}{6} n(n-1)(2n-1)\right) \right)
    \label{eq:subeq in GBS}  
    \\
    = &  \frac{1}{1-c} \left( \left(\frac{4c^2}{\pi^2} - \frac{c^2}{3}\right) n^2 + \left(\frac{c^2}{2} - 2c\right)n - \left(2 - 2c + \frac{c^2}{6}\right)\right).
    \label{eq: GBS lower bound long line 1}
\end{align}
\end{subequations}

To obtain \cref{eq:subeq in GBS}, we remove the two absolute value signs in the RHS of \cref{eq: two abs} by assuming expressions inside the absolute value are non-negative. Note that the term $\frac{4 c^2 n^{2}}{\pi^{2}} - (1-c + cn) $ is a quadratic function in $n$ with positive quadratic coefficient, so for given $c$, we can solve for $n$ such that the expression inside the absolute value is nonnegative. In particular, for any $n$ that satisfies
\begin{align}\label{eq: condition on n}
    8c n \geq \pi^2 \left(1 + \sqrt{1 + (16 (1-c)) / (\pi^2)} \right),
\end{align}
$  (4 c^2 n^{2})/(\pi^{2}) - (1-c + cn)$ is nonnegative. The RHS of \cref{eq: condition on n} is approaching positive infinity when $c$ is approaching $0$, and it is approaching zero when $c$ is approaching $1$.

Assuming \eqref{eq: condition on n} holds, the expression in \eqref{eq:subeq in GBS}, which can be re-written as \cref{eq: GBS lower bound long line 1}, is also a quadratic function in $n$ with positive qudratic coefficient. Therefore, for given $c$, if $n$ satisfies
\begin{align}\label{eq: condition2 on n}
    n \geq \frac{\left(2 c-c^{2}\right)+\sqrt{\left(\frac{c^{2}}{2}-2 c\right)^{2}-4\left( \frac{4c^{2}}{\pi^{2}}-\frac{c^{2}}{3}\right)\left(2-2 c+\frac{c^{2}}{6}\right)}}{\left( \frac{8c^{2}}{\pi^{2}}-\frac{2c^{2}}{3}\right)},
\end{align}
we have that the expression on the RHS of \cref{eq:subeq in GBS} is non-negative. Although \cref{eq: condition2 on n} looks complicated, we observe that when $c$ is relatively large, the RHS of \cref{eq: condition2 on n} is small. For example, when $c=0.9$, the RHS of \cref{eq: condition2 on n} is around $20$. Therefore, assuming \cref{eq: condition on n} and \cref{eq: condition2 on n} hold, we can remove both absolute value signs in \eqref{eq: two abs} and the inequality \cref{eq:subeq in GBS} holds. 

Define the following constants
$$     c_0 = \frac{4}{\pi^2} - \frac{1}{3} = \frac{12-\pi^2}{3\pi^2},\quad 
    c_1 = \frac{c^2}{2} - 2c, \quad
    c_2 = 2 - 2c + \frac{c^2}{6}.$$
Plugging in $\kappa = \frac{1-c+cn}{1-c} = 1 + \frac{c}{1-c}n$ and $n = \frac{(1-c)\kappa -1 + c}{c}$ in \cref{eq: GBS lower bound long line 1}, we can extract an $n$ from \cref{eq: GBS lower bound long line 1} and rewrite it as
\begin{align}
      n \left(\kappa   c c_0  - c c_0  + \frac{c_1}{1-c} - \frac{c_2}{(1-c)n} \right)  = \mathcal{O}(n \kappa c c_0).
\end{align}

Therefore, once we satisfy the condition \cref{eq: condition on n}, we obtain
\begin{align}
    \|r^k\| & \geq (1 - \delta) \left( 1 - \frac{ 3\pi^2}{ (12-\pi^2) n \kappa c }  \right)^{k + 1 } \|r^0\|, \notag \\
    f(x^k) - f^* &\geq (1 - \delta) \left( 1 - \frac{ 3\pi^2}{ (12-\pi^2) n \kappa c }  \right)^{2k + 2 } ( f(x^0) - f^*) . \notag
    \end{align}

\end{proof}

\rvs{
\subsection{Summary of Proof Techniques}

We have adopted four different approaches
for the four desirable bounds
(upper and lower bounds of sGS-CD and GBS-CD).
 As for the single proof, the GBS-CD lower bound
is the most complicated one. 
However, from a global perspective of the whole paper, a
major challenge is to choose the right approach for each of them. 
 We summarize the proof approaches of the four situations
 in Table \ref{tab:summary of proof}.
 


\begin{table}[]
    \centering
\caption{Main proof approaches: high-level method and key inequalities.
 The relations in GBS lower bound
 utilized the special structure of the example;
 other inequalities are generic algebraic inequalities
 that hold for any symmetric $Q$ and any $Z, J$.
  In this table, $Q^{-1} = aI - bJ$, $Z = I - A \Gamma^{-1} A^T $
 and $J = e e^T $ where $e = (1; 1; \dots; 1).$
 }
\renewcommand\arraystretch{1.2}
\begin{tabular}{|c|c|c|}
\hline
    &      upper bound              &   lower bound \\
\hline
    &    reduction of error:        &  i) $M \sim Z^T Z$; \\
sGS &  both forward/backward passes & ii) $\rho( Z^T Z ) = \| Z\|^2 \geq \rho(Z)^2 $;  \\
    &   cause reduction.            &  iii) $\rho(Z) $ can be lower bounded. \\
\hline
    &   bound update matrix:                             &   i) $\rho(M) = 1 - 1/\| \Gamma Q^{-1} \Gamma^{T} \|$  \\
    &  i) $\rho(M) = \| I - \Gamma^{-1} Q \Gamma^{-T} \|$  &   $  \geq 1 - 1/ \textbf{(} a \| \Gamma \Gamma^T  - Q  \|$   \\
GBS &   $=1 - 1/\| \Gamma Q^{-1} \Gamma^{T} \| $;        &   $- a \| Q \| - b \| \Gamma^T J \Gamma   \| \textbf{)}$;   \\
    &  ii) $\| \Gamma Q^{-1} \Gamma^{T}\|$ can be           & ii)  each term in RHS of i)              \\
    &    upper bounded.                                  &   can be  upper or lower bounded.        \\
\hline
\end{tabular}
\label{tab:summary of proof}
\end{table}

}

\section{Numerical Experiments on CD}
\label{sec:experiments}

In this section, we provide some empirical results of the worst-case performance of CD with sGS and GBS update rules by solving unconstrained quadratic problems \rvs{defined in \cref{eq:Ac problem}}.

In the first experiment, we initialize the solutions from the uniform distribution between $0$ and $1$. In \cref{tab:iterations for c equals 0.8}, each column presents the number of epochs needed for an algorithm to solve \cref{eq:Ac problem} 
with various problem sizes $n$ when $c = 0.8$ up to relative error accuracy $1\mathrm{e}{-8}$. 
To make the comparison of various methods more clear, we create
 \cref{tab:ratios for c equals 0.8} based on \cref{tab:iterations for c equals 0.8}.
The columns under ``Ratio of GBS-CD'' of \cref{tab:ratios for c equals 0.8} are created by dividing the number of epochs for GD, C-CD, R-CD, and RP-CD in \cref{tab:iterations for c equals 0.8} by that of GBS-CD.
The resulting ratios can be interpreted as the acceleration ratios of these algorithms over
GBS-CD. For instance, the entry $17360.0 $ in the fourth column of  \cref{tab:ratios for c equals 0.8} means that GBS-CD takes $17360$ times more epochs than R-CD to solve \cref{eq:Ac problem} when $c = 0.8$ and $n = 100$ (or simply put, GBS-CD is $17360$ times slower than R-CD). 
Similarly,  the columns under ``Ratio of GBS-CD'' of \cref{tab:ratios for c equals 0.8} 
display the acceleration ratios of various algorithms over GBS-CD.

Recall we have shown that GBS-CD can be $\mathcal{O}(n)$ times slower than GD in the worst case. In the second column of \cref{tab:ratios for c equals 0.8}, the ratios for $n = 100$ and $n = 600$ are $8.7$ and $51.6$ correspondingly. Note that $\frac{51.6}{8.7} = 5.88 \approx 6$, which matches our theoretical analysis. In the fourth column, we observe that the relative ratio is $\frac{17360.0}{499.3} = 34.8 \approx  36$, 
which also matches our theoretical analysis that GBS-CD is $\mathcal{O}(n^2)$ times slower than R-CD in the worst case. 
In general, \cref{tab:ratios for c equals 0.8} shows that the experimental results match our theoretical analysis.


In the second experiment, we compare the spectral radius of the iteration matrices for various methods. 
In \cref{tab:spectral radius for c equals 0.1 }, we present $1 - \rho(M)$ where $M$ is the (expected) iteration matrix of GBS-CD, sGS-CD, C-CD, R-CD, RP-CD and GD for $c = 0.5, 0.8, 0.99$ and $n = 20, 100, 1000$. The first column shows the value of $c$ in the matrix $Q$ in \cref{eq:Ac problem}, and the next six columns present the values of $1 - \rho(M)$ for each algorithm.
In the last four columns, we divide the values of $1 - \rho(M) $ of C-CD, R-CD, RP-CD and GD by that of GBS-CD to obtain ratios of the spectral radius of the (expected) update matrices. We omit the ratios of various algorithms over sGS-CD, since they are similar to the ones over GBS-CD. We observe that, as $c$ approaches $1$, the gap of the ratios between GD and GBS-CD  grows in $\mathcal{O}(n)$ as $n$ increases, and the gap of the ratios between RP-CD (or R-CD) and GBS-CD grows in $\mathcal{O}(n^2)$. The gap of the ratios between GBS-CD and C-CD is a constant $2$ for different $n$. These observations match the comparison of the number of iterations  in \cref{tab:ratios for c equals 0.8}.

\begin{table}[htbp]
  \centering
  \caption{Comparison of the epochs of GBS-CD, sGS-CD, C-CD, R-CD, RP-CD and GD for solving \cref{eq:Ac problem} when $c = 0.8$. The numbers represent the epochs needed to achieve relative error $1\mathrm{e}{-8}$.}
    \begin{tabular}{|c|c|c|c|c|c|c|}
    \hline
    \textbf{n} & \textbf{GBS-CD}  & \textbf{sGS-CD}  & \textbf{C-CD}   & \textbf{GD} & \textbf{R-CD} & \textbf{RP-CD} \\
    \hline
    100   & 51431 & 51431   & 25749   & 5942   & 103   & 88  \\
    \hline
    200   & 200184 & 200184   & 100377   & 11617   & 103   & 89 \\
    \hline
    600  & 1735996 & 1735996 & 869123 & 33627  & 100  & 88  \\
    \hline
    \end{tabular}%
  \label{tab:iterations for c equals 0.8}%
\end{table}%

\begin{table}[htbp]
  \centering
  \caption{Ratios of the epochs of GBS-CD and sGS-CD over the epochs of C-CD, R-CD, RP-CD and GD for solving \cref{eq:Ac problem} when $c = 0.8$. 
  }
    \begin{tabular}{|c|c|c|c|c|c|c|c|c|}
    \hline
    \multirow{2}[4]{*}{n} & \multicolumn{4}{c|}{Ratio of GBS-CD} & \multicolumn{4}{c|}{Ratio of sGS-CD} \\
\cline{2-9}          & GD & C-CD  & R-CD  & RP-CD & GD & C-CD  & R-CD  & RP-CD \\
    \hline
    100   & 8.7 & 2.0   & 499.3   & 584.4  & 8.7 & 2.0   & 499.3   & 584.4  \\
    \hline
    200   & 17.2 & 2.0   & 1943.5  & 2249.3   & 17.2 & 2.0   & 1943.5  & 2249.3 \\
    \hline
    600  & 51.6 & 2.0  & 17360.0 & 19727.2 & 51.6 & 2.0  & 17360.0 & 19727.2 \\
    \hline
    \end{tabular}%
  \label{tab:ratios for c equals 0.8}%
\end{table}%

\begin{table}[!htbp]
  \centering
  \caption{Comparison of GBS-CD, sGS-CD, C-CD, R-CD, RP-CD and GD for solving \cref{eq:Ac problem} when $c = 0.5,0.8,0.99$}
  \tabcolsep=0.085cm
    \begin{tabular}{|c|c|c|c|c|c|c|c|c|c|c|}
    \hline
    \multirow{2}[4]{*}{c} & \multicolumn{6}{c|}{1 - $\rho$(M)}               & \multicolumn{4}{c|}{Acceleration Ratio} \\
\cline{2-11}          & GBS-CD & sGS-CD & GD    & C-CD  & R-CD  & RP-CD & GD    & C-CD  & R-CD  & RP-CD \\
    \hline
    \multicolumn{11}{|c|}{n = 20} \\
    \hline
    0.5   & 4.2e-2 & 4.2e-2 & 4.8e-2 & 7.6e-2 & 4.0e-1 & 5.2e-1 & 1.1   & 1.8   & 9.3   & 12.1 \\
    \hline
    0.8   & 7.4e-3 & 7.4e-3 & 1.2e-2 & 1.4e-2 & 1.8e-2 & 2.0e-1 & 1.6   & 1.9   & 2.4   & 26.8 \\
    \hline
    0.99  & 2.5e-4 & 2.5e-4 & 5.0e-4 & 4.9e-4 & 1.0e-2 & 1.0e-2 & 2.0   & 2.0   & 40.0  & 41.2 \\
    \hline
    \multicolumn{11}{|c|}{n = 100} \\
    \hline
    0.5   & 1.9e-3 & 1.9e-3 & 9.9e-3 & 3.8e-3 & 3.9e-1 & 5.0e-1 & 5.1   & 2.0   & 202.1 & 259.1 \\
    \hline
    0.8   & 3.0e-4 & 3.0e-4 & 2.5e-3 & 6.4e-4 & 1.8e-1 & 2.0e-1 & 8.1   & 2.1   & 586.3 & 651.5 \\
    \hline
    0.99  & 1.0e-5 & 1.0e-5 & 1.0e-4 & 2.0e-5 & 1.0e-2 & 1.0e-2 & 10.0  & 2.0   & 1e3 & 1e3 \\
    \hline
    \multicolumn{11}{|c|}{n = 1000} \\
    \hline
    0.5   & 1.9e-5 & 1.9e-5 & 9.9e-4 & 3.9e-5 & 3.9e-1 & 5.0e-1 & 50.7  & 2.0   & 1.9e4 & 2.5e4 \\
    \hline
    0.8   & 3.0e-6 & 3.0e-6 & 2.5e-4 & 6.2e-6 & 1.8e-1 & 2.0e-1 & 81.2  & 2.0   & 5.8e4 & 6.4e4 \\
    \hline
    0.99  & 1.0e-7 & 1.0e-7 & 1.0e-5 & 2.0e-7 & 1.0e-2 & 1.0e-2 & 100.0 & 2.0   & 9.9e4 & 9.9e4 \\
    \hline
    \end{tabular}%
  \label{tab:spectral radius for c equals 0.1 }%
\end{table}%




\section{Symmetrization Rules for ADMM}\label{sec: ADMM results}

\subsection{Motivation in ADMM}

\rvs{In previous sections, we have analyzed the worst-case performance of GBS-CD and sGS-CD in solving unconstrained quadratic problems. In this section,
we will discusss the performance of ADMM
with symmetrized update rules GBS and sGS.}

To solve large-scale problems with linear constraints, a natural idea is to combine CD methods with augmented Lagrangian method to obtain the so-called \rvs{ADMM algorithms}
\cite{glowinski1975approximation,chan1978finite,gabay1976dual,boyd2011distributed}.
Unlike CD where any reasonable update order can lead to convergence \cite{tseng2001convergence},
 for multi-block ADMM, even the most basic cyclic version does not converge \cite{chen2016direct}. 
 Small step-size versions of multi-block ADMM can be shown to converge with extra assumptions on the problem (see, e.g.
\cite{hong2012linear, han2012note,chen2013convergence,he2013proximal,he2013full, lin2014convergence,hong2014block,cai2014direct,sun2014convergent,lin2014global,han2014augmented,li2014schur,li2015convergent,lin2015iteration,deng2017parallel}), but the lesson from CD methods is that the speed advantage of CD exactly comes from large stepsize, thus we are more interested in ADMM with large dual step-size (such
as dual step-size $1$).




We are only aware of three major variants of multi-block ADMM with dual stepsize $1$ that are convergent
 in numerical experiments: Gaussian back substitution ADMM (GBS-ADMM) \cite{he2012alternating,he2012convergence},
symmetric Gauss-Seidel ADMM (sGS-ADMM) \cite{li2016schur,chen2017efficient} and randomly permuted ADMM (RP-ADMM) \cite{sun2015expected}. 
The first two use deterministic orders, and the third uses a random order. 
\rvs{It is known that}
the theoretical analysis of random permutation is notoriously difficult even for CD and SGD \cite{recht2012beneath,sun2015expected,wright2017analyzing,lee2018random,gurbuzbalaban2015random,gurbuzbalaban2018randomness},
and for RP-ADMM only the expected convergence for quadratic objective function is \rvs{given} \cite{sun2015expected, chen2015convergence}.
In contrast, GBS-ADMM enjoys strong theoretical guarantee as the convergence for separable 
convex objective with linear constraints is proved \cite{he2012alternating}.
For sGS-ADMM, the convergence guarantee is proved for a sub-class of convex problems. 

If our purpose is just to resolve the divergence issue of cyclic ADMM,
then GBS-ADMM and sGS-ADMM both provide rather satisfactory (though not perfect) theoretical guarantee on the convergence.
However, the major purpose of using block decomposition is to solve large-scale problems, thus
the convergence speed is also very important.
What can we say about the convergence speed of GBS-ADMM and sGS-ADMM? 
Are they provably faster the one-block version, just like R-CD \cite{nesterov2012efficiency, leventhal2010randomized}? 
\rvs{Previously, we have analyzed GBS and sGS update orders for unconstrained problems, and proved that they do not improve
the convergence rate of cyclic order in the worst case.}
\rvs{In the rest of the paper, we will present an example
that the convergence speed of sGS-ADMM and GBS-ADMM is
much slower than the one-block version, and RP-ADMM.}

\subsection{Two ADMM Algorithms}

To solve the following linearly constrained problem 
\begin{equation}\label{eq: general constrained problem}
\begin{split}
    \min_{x \in \dR^n }  \quad   & f(x_1, \cdots, x_n) \\
    \text{s.t.} \quad & \sum_i A_i x_i  = b,
\end{split}
\end{equation}
we consider the augmented Lagrangian function
\begin{equation}\label{augmented Lag func}
\mathcal{L}(x_1, \dots, x_n ; \lambda ) = f(x_1, \dots, x_n) - \lambda^T (\sum_i A_i x_i  - b) + \frac{\sigma}{2} \| \sum_i A_i x_i - b \|^2.
\end{equation}

\begin{algorithm}
\caption{\normalsize sGS-ADMM}
\label{Algorithm: n-block sGS-ADMM}
\begin{algorithmic}[1]
\FOR{$k=0,1,2,\ldots,$}
    \STATE  \textit{Forward Pass}:
    \STATE $\quad x_{1}^{k+\frac{1}{2}} \in \argmin_{x_{1}} \mathcal{L}\left(x_{1}, x_{2}^{k-1}, x_{3}^{k-1}, \ldots x_{n}^{k-1}; \lambda^k \right)  +  \frac{\sigma}{2}\|x_{1} - x_{1}^{k+\frac{1}{2}}\|^2_{T_1}$
    \STATE $\quad \cdots$
    \STATE $\quad x_{n}^{k+\frac{1}{2}} \in \argmin_{x_{n}} \mathcal{L}\left(x_{1}^{k+\frac{1}{2}}, x_{2}^{k+\frac{1}{2}}, \ldots x_{n};  \lambda^k \right) + \frac{\sigma}{2}\|x_{n} - x_{n}^{k+\frac{1}{2}}\|^2_{T_n}$
    \STATE  \textit{Backward Pass}:
    \STATE $\quad x_{n}^{k+1}  = x_{n}^{k+\frac{1}{2}}$
    \STATE $\quad x_{n-1}^{k+1}  \in \argmin_{x_{n-1}} \mathcal{L}\left(x_{1}^{k+\frac{1}{2}}, \ldots, x_{n-1}, x_{n}^{k+1};  \lambda^k \right) + \frac{\sigma}{2}\|x_{n-1} - x_{n-1}^{k+1}\|^2_{T_{n-1}} $
    \STATE $\quad \cdots$
    \STATE $\quad x_{1}^{k+1} \in \argmin_{x_{1}} \mathcal{L}\left(x_{1}, x_{2}^{k+1}, x_{3}^{k+1}, \ldots x_{n}^{k+1};  \lambda^k \right) + \frac{\sigma}{2}\|x_{1} - x_{1}^{k+1}\|^2_{T_1}$
    \STATE  \textit{Dual Update}: $\lambda^{k+1} = \lambda^k - \beta(  A_1 x_1^{k+1} + \dots +  A_n x_n^{k+1} - b ).$
\ENDFOR
\end{algorithmic}
\end{algorithm}

The general sGS-ADMM is defined in \cref{Algorithm: n-block sGS-ADMM}  for solving
the constrained problem \cref{eq: general constrained problem}, where each $T_i$ is a self-adjoint positive semidefinite linear operator that satisfies the conditions mentioned in \cite{li2016schur}. 
As our goal is to understand the worst-case convergence rate,
we consider  a special setting that $T_i = 0, \ \forall i$,
which for the unconstrained problems reduces to the sGS-CD \cref{Algorithm: n-block sGS-CD}. 
In the section of experiments, we will consider sGS-ADMM with $T_i = 0, \forall i$, for linearly constrained problems. 







When applying GBS update order to ADMM for solving linearly \rvs{constrained} problem \cref{eq: general constrained problem}, the prediction step 
is a regular primal update of Cyclic ADMM (C-ADMM). 
After the prediction step, GBS-ADMM updates the dual variable using the predicted primal variable $\tilde{x}^k$. 

To derive the correction step of GBS-ADMM, we first denote $\Omega \triangleq A^T A$ where $A$ is the linear constraint matrix in problem \cref{eq: general constrained problem}, and $\Gamma_{\Omega}$ is the lower triangular matrix of $\Omega$ (with the diagonal entries).

We define the correction matrix $F$ as
\begin{align}\label{eq: matrix F in GBS-ADMM}
 F \triangleq  \begin{bmatrix}
1 &   0 \\
0 &[\Gamma_{\Omega}^{-T}]_{2:n}
\end{bmatrix}
\begin{bmatrix}
1 &   0 \\
0 & [D_{\Omega}]_{2:n}
\end{bmatrix}, 
\end{align}
where 
$\Gamma_{\Omega}^{-T}$ is the transpose of the inverse of $\Gamma_{\Omega}$ and $D_{\Omega}$ is the diagonal matrix of $\Omega$. $[\Gamma_{\Omega}^{-T}]_{2:n}$ and $[D_{\Omega}]_{2:n}$ are the sub-matrices by excluding the first row and first column of $\Gamma_{\Omega}^{-T}$ and $D_{\Omega}$ respectively. 
Detailed GBS-ADMM algorithm is presented in \cref{Algorithm: n-block GBS-ADMM}.

\begin{rmk}
Strictly speaking, GBS-CD (defined in \cref{Algorithm: n-block GBS-CD}) is not a special form of GBS-ADMM (defined in \cref{Algorithm: n-block GBS-ADMM}) to solve unconstrained problems. Although the correction matrix $B$ \cref{eq: matrix B in GBS-CD} and $F$ \cref{eq: matrix F in GBS-ADMM} for GBS-CD and GBS-ADMM are in the same form, they are different: $B$ is constructed from $Q=A^TA$, where $A$ appears in the quadratic objective of \cref{least square optimization formulation for sGS-CD}, but $F$ is constructed from $\Omega=A^TA$ where $A$ is the matrix of linear constraint in \cref{eq: general constrained problem}.
Note that the original GBS-ADMM is only defined
for a separable convex objective function.
We presented a version for general convex objective,
but as our goal is to reveal the limitation of GBS-ADMM,
we do not need to consider the most general form, but only need
to study some special forms (objective is zero or quadratic). 

If the objective function is quadratic as in \cref{least square optimization formulation for sGS-CD}, we can define a new
version of GBS-ADMM (which can be called obj-GBS-ADMM) as follows: for the correction step, we construct the matrix $F$ from the matrix in the objective function instead of from the matrix in the linear constraint. GBS-CD defined in \cref{Algorithm: n-block GBS-CD} can be viewed as a special form of this obj-GBS-ADMM for unconstrained quadratic problems.

For theoretical analysis, we prove that GBS-CD can be very slow.
For GBS-ADMM, we provide numerical experiments to show that it can be very slow for a bad example.
 The theoretical evidence for GBS-CD and the empirical evidence for GBS-ADMM together indicate that GBS-ADMM can be very slow in the worst case.
%
 
%



\end{rmk}

\begin{algorithm}
\caption{\normalsize GBS-ADMM }
\label{Algorithm: n-block GBS-ADMM}
\begin{algorithmic}
\FOR{$k=0,1,2,\ldots,$}
    \STATE  \textit{Prediction Step}:
    \STATE $\quad \tilde{x}_{1} \in \argmin_{x_{1}} \mathcal{L}\left(x_{1}, x_{2}^{k-1}, x_{3}^{k-1}, \ldots x_{n}^{k-1}; \lambda^k \right)$
    \STATE $\quad \cdots$
    \STATE $\quad \tilde{x}_{n} \in \argmin_{x_{n}} 
    \mathcal{L}\left(\tilde{x}_{1}^{k-1}, \tilde{x}_{2}^{k-1}, \ldots
    \tilde{x}_{n-1}^{k-1}, x_{n}; \lambda^k \right)$
    \STATE  \textit{Dual Update}: $\tilde{\lambda}^{k} = \lambda^k - \sigma (  A_1 \tilde{x}_{1} + \dots +  A_n \tilde{x}_{n} - b ) $
    \STATE  \textit{Correction Step}:
    \STATE $\quad x^{k+1} = x^k - \beta F(x^k - \tilde{x}^k)$
    \STATE $\quad \lambda^{k+1} = \lambda^k - \beta (\tilde{\lambda}^{k+1} - \lambda^k)$
\ENDFOR
\end{algorithmic}
\end{algorithm}

\subsection{Numerical Experiments on ADMM}
\label{sec:experiments ADMM}

In this section, we provide some empirical results of the worst-case performance of CD with sGS and GBS update rules by solving unconstrained quadratic problems and and linearly constrained problems.



We consider solving a problem with a strongly convex quadratic objective and a linear constraint:
\begin{align}\label{eq: solve strongly convex}
    \min_{x \in \dR^n }  \quad   x^T Q x \quad
    \text{s.t.} \quad & Qx = 0, 
\end{align}
where $Q$ is defined by \cref{eq:Ac def}, in which $c=0.95$.

In \cref{tab:iterations for strongly convex ADMM for c equals 0.3}, we list the number of epochs needed for each algorithm to obtain relative error $1\mathrm{e}{-5}$. We observe that
with strongly convex objective, GBS-ADMM and sGS-ADMM still converge much slower than RP-ADMM or ALM. We also compare the ratios of the number of epochs of RP-ADMM and ALM over GBS-ADMM (and sGS-ADMM) in \cref{tab: ratios for strongly convex ADMM c = 0.3}. 
\rvs{From \cref{tab: ratios for strongly convex ADMM c = 0.3}, we observe that GBS-ADMM and sGS-ADMM are $\mathcal{O}(n)$ times slower than ALM in this example. However, they are not exactly $\mathcal{O}(n^2)$ times slower than RP-ADMM:
the gap is slightly smaller than $\mathcal{O}( n^2 )$.
We make the following conjecture:
\conjecture{In solving the problem \cref{eq: solve strongly convex}, GBS-ADMM and sGS-ADMM can be $\mathcal{O}(n^2 / \log n)$ times slower than RP-ADMM.}
}
\begin{rmk}
\rvs{ The $\log n $ term also appears in the literature which uses the same special matrix \cref{eq:Ac def} to demonstrate the worst-case performance of C-CD \cite{sun2016worst}. 
}
\end{rmk}

\begin{table}[ht]
  \centering
  \caption{Comparison of GBS-ADMM, sGS-ADMM, RP-ADMM, and ALM for solving \cref{eq: solve strongly convex} where $c = 0.3$. The numbers represent the number of epochs to achieve relative error $1\mathrm{e}{-5}$.}
    \begin{tabular}{|c|c|c|c|c|}
    \hline
    \textbf{n} & \textbf{GBS-ADMM} & \textbf{sGS-ADMM} & \textbf{RP-ADMM} & \textbf{ALM} \\
    \hline
    100   & 10709 & 13326 & 95   & 3696 \\
    \hline
   200   & 54932 & 71877 & 126  & 8607 \\
    \hline
    400   & 266344 & 391654 &  168  & 18044 \\
    \hline
    \end{tabular}%
  \label{tab:iterations for strongly convex ADMM for c equals 0.3}%
\end{table}

\begin{table}[H]
  \centering
  \caption{Ratios of the epochs of GBS-ADMM and sGS-ADMM over the epochs of RP-ADMM, and ALM for solving \cref{eq: solve strongly convex} when $c = 0.3$}
    \begin{tabular}{|c|c|c|c|c|}
    \hline
         & \multicolumn{2}{c|}{Ratio of GBS-ADMM} & \multicolumn{2}{c|}{Ratio of sGS-ADMM} \\
    \hline
    n     & RP-ADMM & ALM   & RP-ADMM & ALM \\
    \hline
   100   & 112.7 & 2.9   & 140.2 & 3.6 \\
    \hline
    200   & 436.0 & 6.4   & 570.5 & 8.4 \\
    \hline
    400   &1585.4  & 14.7  & 2331.3 & 21.7 \\
    \hline
    \end{tabular}%
  \label{tab: ratios for strongly convex ADMM c = 0.3}
\end{table}%

    \rvs{
    \begin{rmk}
One may wonder whether the gap can be even larger than $\mathcal{O}(n^2)$.
Based on our experiments of other problem instances, we suspect
that the gap is at most $\mathcal{O}(n^2 )$ or $\mathcal{O}(n^2/\log n)$.
In particular, we have performed two more experiments.
In the first experiment, 
we use the same special matrix $Q$ defined in \cref{eq:Ac def}
but vary the value of $c$. When $c$ is small, we observe that the gap between
    GBS-ADMM (or sGS-ADMM) and RP-ADMM becomes $\mathcal{O}(n)$,
which is much smaller than $\mathcal{O}(n^2/\log n)$.
In the second experiment,
we consider other problem data, such as randomly generated tridiagonal matrices and circulant Hankel matrices. 
In these cases, the performance of GBS-ADMM is comparable to RP-ADMM, and sGS-ADMM is the fastest one. Again, the gap between GBS-ADMM (or sGS-ADMM) and RP-ADMM 
is much smaller than $\mathcal{O}(n^2/ \log n)$.
See \cref{appendix: additional experiments} for the details
of these experiments. 
\end{rmk}
}

\section{Discussions and Conclusions}
\label{sec:conclusions}

\rvs{\subsection{Discussions of practical performance and multi-block ADMM.}



We remark that the worst-case slow convergence of an algorithm does not necessarily imply the slow convergence of an algorithm
in practice.
For instance, C-CD is shown to be up to $\mathcal{O}(n)$ times slower than GD in the worst case \cite{sun2016worst}, but 
\rvs{for many practical problems, C-CD is faster than GD \cite{sun2015improved, sun2016worst, li2017faster}}.
It is an interesting open problem to explain the practical behavior of C-CD. 
Similarly, the slow convergence of algorithms with sGS and GBS orders does not imply that they are slow for practical problems, but knowing the worst-case performance provides
better understanding of these algorithms.

As for ADMM,
our findings imply a significant gap: 
despite the extensive studies on ADMM, none of the existing
variants of multi-block ADMM can inherit the advantage of R-CD: an improvement ratio of $O(1) $ to $\mathcal{O}(n)$ in convergence speed
compared to the single-block method.
Before our paper, it seems that sGS-ADMM and GBS-ADMM are the closest to this goal, but our results provide strong evidence that
they are slow in the worst case. 
A remaining candidate for this goal is RP-ADMM, but this is a difficult task because even for RP-CD  the precise convergence speed remains unproved.
There are a few ways to fill in this theoretical gap: proposing a new method that achieves this goal,
or proposing a new framework to justify the advantage of deterministic ADMM (which would justify the advantage of
deterministic CD), or justifying the advantage of RP-ADMM. Each of these solutions would be rather non-trivial and quite
interesting. 

}

\subsection{Conclusions}

In this paper, we study the worst-case convergence rate of two symmeterized orders sGS and GBS for CD and ADMM.
For ADMM, these two update orders are among the most popular variants,
and also have strong convergence guarantee.
In this work, we prove that for unconstrained problems,
sGS-CD and GBS-CD are $\mathcal{O}(n^2)$ times slower than R-CD in the worst case. 
In addition, we show empirically that when solving quadratic problems with linear constraints, sGS-ADMM and GBS-ADMM can be roughly $\mathcal{O}(n^2)$
times slower than randomly permuted ADMM for a certain example.
These results indicate that the symmetrization trick does not resolve the slow worst-case
convergence speed of deterministic block-decomposition methods.
Technically, we provide a unified framework of symmetrization, which
includes sGS-CD and GBS-CD as special cases. This framework can help
understand algorithms from the perspective of update matrices.


\appendix
\newpage

\section{Proofs of the Two Upper Bounds}\label{app: proof of upp bounds}

\subsection{Proof of the Upper Bound  of  sGS-CD (\cref{prop: upper bound sGS-CD})}\label{subsec: prop: upper bound sGS-CD}

\begin{proof}


\rvs{We first assume $Q$ is positive definite. 
As shown in \cref{eq: sGS-CD update matrix for iterate}, the update matrix of sGS-CD  is
\[(I - \Gamma^{-T}Q)(I - \Gamma^{-1}Q).\]
}

By defining $x^{k+1/2} = (I - \Gamma^{-1}Q)x^k$ as the ``half step" update, we separate the effects of forward and backward pass. Using the techniques developed in \cite{sun2016worst}, we can obtain an upper bound on the decreases of objective error from $x^k$ to $x^{k+1/2}$ for forward pass (and from  $x^{k+1/2}$ to $x^{k+1}$ for backward pass). In particular, we can derive the following two inequalities from either optimization perspective or matrix recursion perspective, as indicated in \cite{sun2016worst}:

\begin{align}
    & \frac{f(x^k) - f(x^*)}{ f(x^k) - f(x^{k+\frac{1}{2}})}
      \leq  \| D_Q^{-1/2} \Gamma^T Q^{-1} \Gamma D_Q^{-1/2}\| \triangleq c_1 \label{eq: B_1 line 1} \\
    & \frac{f(x^{k+1}) - f(x^*)}{ f(x^{k+1}) - f(x^{k+\frac{1}{2}})}
      \leq  \| D_Q^{-1/2} U^T Q^{-1} U D_Q^{-1/2}\| \triangleq c_2 \label{eq: B_1 line 2}.
\end{align}

The proof of \cref{eq: B_1 line 1} can be found in the proof of Claim B.1 in \cite{sun2016worst}, and \cref{eq: B_1 line 2} can be proved in exactly the same way. \cref{eq: B_1 line 1} and \cref{eq: B_1 line 2} imply
\begin{align}
   & f\left(x^{k+\frac{1}{2}}\right)-f\left(x^{*}\right) \leq\left(1-\frac{1}{c_1}\right)\left(f\left(x^{k}\right)-f\left(x^{*}\right)\right) \label{eq: B_3},\\
   & f\left(x^{k+1}\right)-f\left(x^{*}\right) \leq\left(1-\frac{1}{c_2}\right)\left(f\left(x^{k+\frac{1}{2}}\right)-f\left(x^{*}\right)\right).\label{eq: B_4}
\end{align}
Combining \eqref{eq: B_3} and \eqref{eq: B_4}, we obtain
\begin{align}
    f(x^{k+1}) - f(x^*) \leq \left(1- \frac{1}{c_2}\right) \left(1- \frac{1}{c_1}\right) \left(f(x^k) - f(x^*)\right). \label{bound GBS}
\end{align}

Therefore, we can obtain an upper bound on the objective error convergence rate of GBS-CD by finding the upper bounds for $c_1$ and $c_2$. Notice that $Q$ is symmetric, and thus we observe that $c_1$ equals to $c_2$ from their definitions in \cref{eq: B_1 line 1} and \cref{eq: B_1 line 2}. 

According to the fact that $\left\|D_{Q}^{-1 / 2} B D_{Q}^{-1 / 2}\right\| \leq \frac{1}{\min _{i} Q_{ii}}\|B\|=\frac{1}{L_{\min }}\|B\|$ for any positive definite matrix $B$, where $L_{\min }  \triangleq \min_{i} Q_{ii}$, we have
\begin{align}
    c_1 = \left\|D_{Q}^{-1 / 2} \Gamma^{T} Q^{-1} \Gamma D_{Q}^{-1 / 2}\right\| \leq \frac{1}{L_{\min }}\left\|\Gamma^{T} Q^{-1} \Gamma\right\|. \label{bound c1}
\end{align}

We then apply \cref{claim B2}, which states that 
\begin{align}
\|\Gamma^T Q^{-1} \Gamma\| \leq \frac{1}{L_{\min} } \kappa \cdot \min \left\{   \sum_i L_i, (2 + \frac{1}{\pi} \log n)^2  L    \right\}. \label{bound A.1}
\end{align}

Finally, combining \eqref{bound GBS}, \eqref{bound c1}, \eqref{bound A.1}, the fact that $c_1 = c_2$, and replacing $\sum_{i} L_{i}$ by $n L_{\mathrm{avg}}$, we obtain the desired results in \cref{eq: SGSCD upper bound prop}:
\[
     f(x^{k+1}) - f^* \leq \left(\min \left\{ 1 - \frac{1}{ n \kappa  } \frac{L_{\min}}{L_{\avg}} ,  1 - \frac{L_{\min} }{ L (2 + \log n/ \pi)^2 } \frac{1}{ \kappa  } \right\}\right)^2  (f(x^k) - f^*).
 \]

If $Q$ is positive semi-definite, then we replace $Q^{-1}$ by $Q^{\dag}$ which is the pseudo-inverse of $Q$. 
The rest of the analysis is similar to the proof in \cite[Proposition 1]{sun2016worst} and omitted. 
\end{proof}

\subsection{Proof of Upper  bound  of  GBS-CD (\cref{prop: upper bound GBS-CD objective})}\label{subsec: prop: upper bound GBS-CD objective}

\begin{proof}
\rvs{
Without loss of generality, we can assume $x^*=0$.
Notice that minimizing $f(x) = x^T Q x - 2 b^T x$ is equivalent to minimizing $ f(x) = (x - x^*)^T Q (x - x^*) $
where $x^* = Q^{\dag}b $ by using the fact that $ Q x^* = Q Q^{\dag}b  = b $ when $b \in \mathcal{R}(Q)$.
By a linear transformation $z = x-x^*$,
minimizing $ (x - x^*)^T Q (x - x^*) $ starting from $x^0$
is equivalent to minimizing $ z^T Q z $ starting from $z^0 = x^0-x^*$.
Thus we can assume $x^* = 0$, or equivalently, $b = 0$.
}

We want to compute the convergence rate of $$ f(x^k) = (x^k)^T Q(x^k) = \| r^k \|^2,$$
where $r^k = A x^k $ in which $A$ satisfies $Q = A^T A$. 

Given $x^{k+1} = (I - B\Gamma ^{-1}Q)x^k$, we have:
\begin{align}\label{eq: GBS residual equations}
    r^{k+1} & = A x^{k+1} = A (I - B\Gamma ^{-1}Q)x^k \notag\\
            & = Ax^k - AB\Gamma ^{-1} A^T (A x^k) \quad  (\textrm{by } Q = A^TA) \notag\\
            & = (I - A B \Gamma ^{-1} A^T)r^k \notag
\end{align}

\rvs{
We observe that $A B \Gamma ^{-1} A^T$ is similar to $B \Gamma ^{-1} Q$. By \cref{prop: GBS-CD symmetrization}, we have
\begin{equation}
    \eig(I - B \Gamma ^{-1} Q) = \eig(I - \Gamma^{-1} Q \Gamma^{-T}).
\end{equation}

Therefore, to analyze $\rho(I - A B \Gamma ^{-1} A^T)$, it is sufficient to analyze $\rho(I - \Gamma^{-1} Q \Gamma^{-T})$.
}

\begin{align}
    & \rho(I - \Gamma^{-1} Q \Gamma^{-T}) 
    =  1 - \lambda_{\min}(\Gamma^{-1} Q \Gamma^{-T}) 
    =  1 - \frac{1}{\rho(\Gamma^{T} Q^{-1} \Gamma)} \notag\\
    = & 1 - \frac{1}{\|\Gamma^{T} Q^{-1} \Gamma\|} 
    \leq 1 - \frac{1}{\kappa \cdot \min \left\{   \sum_i L_i, (2 + \frac{1}{\pi} \log n)^2  L    \right\}} \notag.
\end{align}
These lines follow from the fact that the spectral radius and the spectral norm of a symmetric matrix are the same. The last inequality is obtained by applying the upper bound of $\|\Gamma^T Q^{-1} \Gamma\|$ from \cref{claim B2}. This implies
\begin{align}
    \|r^{k+1}\| & \leq 
    \left(1 - 
    \frac{1}{\kappa \cdot \min \left\{  \sum_i L_i, (2 + \frac{1}{\pi} \log (n))^2  L    \right\}} 
    \right)
    \|r^{k}\|, \notag \\
    f(x^{k+1}) - f^* & \leq 
    \left(1 - 
    \frac{1}{\kappa \cdot \min \left\{  \sum_i L_i, (2 + \frac{1}{\pi} \log (n))^2  L    \right\}} 
    \right)^2
    (f(x^{k}) - f^*) \notag.
\end{align}
\end{proof}


\section{An Alternating Proof of the Lower Bound of sGS-CD}\label{sec: sGS projection proof}

In this proof, we express sGS-CD as alternating projections, then the iteration matrix of sGS-CD becomes a symmetrization of the C-CD iteration matrix. This builds a link with C-CD method and makes the computation of spectral radius feasible.

\begin{proof}
We need to lower bound the objective error convergence rate of sGS-CD, using example \cref{eq:Ac problem} introduced in \cref{subsec: worst case}. In particular, we want to prove a lower bound of the convergence rate of $ f(x^k) = (x^k)^T Qx^k = \| r^k \|^2 $ where $r^k = A x^k$ and $A$ satisfies $A^T A = Q$. 
    We introduce alternating projections and use them to find a simpler form of the iteration matrix of $r^k$.
    
When solving the \cref{eq:Ac problem} using coordinate descent, the update rule for coordinate $i$ is given by
\begin{align}
    x_{i}^{+}=\frac{1}{A_{i}^{T} A_{i}}\left[A_{i}^{T}\left(-A_{-i} x_{-i}\right)\right], \label{eq:cd update rule has plus}
\end{align}
where $A_{-i}$ contains all columns of $A$ except $A_i$, $x_{-i}$ contains all elements of $x$ except $x_i$ and represents the current values, and $x_{i}^{+}$ represents the new value. 
    
Since $r = Ax$, we can rewrite \cref{eq:cd update rule has plus} as
    \begin{subequations}\label{eq:update rule proj 1}
    \begin{align*} 
    x_{i}^{+} &=x_{i}-\frac{1}{A_{i}^{T} A_{i}} A_{i}^{T} r \\ 
    r^{+} &=r+A_{i}\left(x_{i}^{+}-x_{i}\right).
    \end{align*}
    \end{subequations}
    
We define $P_{i}=I-(A_{i}^{T} A_{i})^{-1} A_{i} A_{i}^{T}$ as the projection matrix that projects vectors onto the column space of $A_i$. 
The update rule for $r$ is therefore
    \begin{align*}
        r^{+}& =\left(I-\frac{A_{i} A_{i}^{T}}{A_{i}^{T} A_{i}}\right) r=P_{i} r \label{eq:CD update rule in projection}.
    \end{align*}
We denote $r^{k + \frac{1}{2}}$ to be the value of $r$ after the forward pass of sGS-CD at iteration $k+1$, and it can be expressed as
    \begin{align}
         r^{k + \frac{1}{2}} = P_{n} P_{n-1} \ldots P_{1} r^{k } \notag
    \end{align}
    
Similarly, $r^{k + 1}$, which is the value of $r$ after the backward pass of sGS-CD at iteration $k+1$, can be expressed as
    \begin{align}
         r^{k+1} = P_{1} P_{2} \ldots P_{n-1} r^{k + \frac{1}{2}}. \notag
    \end{align}
    
Combining the forward pass and backward pass, we have matrix recursion for $r^{k+1}$ of sGS-CD:
    \begin{equation}\label{eq: sGS-CD residual update rule}
         r^{k+1} = P_{1} P_{2} \ldots P_{n-1} P_{n} P_{n-1} \ldots P_{1} r^{k }
    \end{equation}

Using the property of projection matrix $P_i = P_i P_i$, \cref{eq: sGS-CD residual update rule} can be rewritten as
    \begin{align}\label{eq:rewrite r using symmetric}
        r^{k+1} = \left(P_n P_{n-1} \cdots  P_1\right)^T \left( P_n P_n P_{n-1} \cdots  P_1 \right)r^{k }.
    \end{align}
    
Note that a forward pass of sGS-CD is equivalent to one full iteration of C-CD. If we denote $\hat{P} = P_n P_{n-1} \cdots  P_1$, then $\hat{r}^{k+1}$ of C-CD at iteration $k+1$ can be expressed as
    \begin{align*}\label{eq: C-CD residual update}
         \hat{r}^{k+1} = P_n P_{n-1} \cdots  P_1 \hat{r}^{k} = \hat{P} \hat{r}^{k}. 
    \end{align*}

The update of $r^{k+1}$ for sGS-CD at iteration $k+1$ is
    \begin{equation}\label{eq: relation: square in sGS-CD residual}
        r^{k+1} = \hat{P}^T \hat{P} r^k.
    \end{equation}
In other words, \cref{eq: relation: square in sGS-CD residual} shows that the iteration matrix of $r^k$ for sGS-CD is a \textbf{``product symmetrization''} of the iteration matrix of $\hat{r}^{k}$ for C-CD. This implies we can apply results of the lower bound of C-CD and the rest of the proof is the same as the one in \cref{subsec: proof sGS lower bound}.

\end{proof}

\section{Intermediate Results in the Proof of the Upper Bounds}

\subsection{Proposition of $ \|  \Gamma^T A^{-1} \Gamma \|$}
\label{appendix: proof of tau A tau}
\begin{proposition}[Claim B.2 in \cite{sun2016worst}]\label{claim B2}
Let $A$ be a positive definite matrix with condition number $\kappa$ and $\Gamma$ is the lower triangular matrix of $A$. Then
    \begin{equation}\label{bound G'AinvG}
    \|  \Gamma^T A^{-1} \Gamma \| \leq
     \kappa \cdot \min \left\{   \sum_i L_i, (2 + \frac{1}{\pi} \log n)^2  L    \right\}.
  \end{equation}
\end{proposition}

\begin{proof}

 Denote $$ \Gamma_{\mathrm{unit} } = \begin{bmatrix}
   1 &  0 & 0 & \dots & 0  \\
   1 &  1 & 0 & \dots & 0  \\
   \vdots & \vdots & \vdots & \ddots & \vdots  \\
   1 &  1 & 1 & \dots & 0  \\
   1 &  1 & 1 & \dots & 1  \\
     \end{bmatrix}  , $$
     then $ \Gamma = \Gamma_{\mathrm{unit} } \circ A$, where $\circ $ denotes the Hadamard product.
According to the classical result on the operator norm of the triangular truncation operator
\cite[Theorem 1]{angelos1992triangular}, we have
$$
  \| \Gamma  \| = \| \Gamma_{\mathrm{unit}} \circ A \| \leq ( 1 + \frac{1}{\pi} + \frac{1}{\pi} \log n  ) \| A \|
   \leq (2 + \frac{1}{\pi} \log n) \| A \|.
$$
Thus we can prove the second part of \eqref{bound G'AinvG} by
\begin{align}
   &  \|  \Gamma^T A^{-1} \Gamma \|
 \leq \| \Gamma^T \Gamma \| \| A^{-1} \|  = \| \Gamma \|^2 \frac{1}{\lambda_{\min}(A) } \\
  \leq  & (2 + \frac{1}{\pi} \log n)^2 \frac{\| A\|^2}{\lambda_{\min}(A)}
  = (2 + \frac{1}{\pi} \log n)^2 \kappa L . 
\end{align}

We can bound $ \| \Gamma \|^2 $ in another way (denote $\lambda_i$'s as the eigenvalues of $A$):
\begin{equation}\label{Gamma square bound, another}
\begin{split}
  \| \Gamma \|^2 \leq \|\Gamma \|_F^2 =  \frac{1}{2}( \|A \|_F^2 + \sum_i A_{ii}^2 )
= \frac{1}{2}  \left( \sum_i \lambda_i^2 +  \sum_i A_{ii}^2  \right)   \\
                 \leq \frac{1}{2} \left( (\sum_i \lambda_i)\lambda_{\max} + L_{\max} \sum_i A_{ii}    \right)
                  \overset{\text{(i)}}{=} \frac{1}{2} ( L + L_{\max} )\sum_i L_i \leq L \sum_i L_i.
       \end{split}
\end{equation}
where(i) is because
$ \sum_i \lambda_i = \text{tr}(A) = \sum_i A_{ii} $ and $A_{ii} = L_i $.
Thus
$$
  \|  \Gamma^T A^{-1} \Gamma \|
 \leq  \| \Gamma \|^2 \frac{1}{\lambda_{\min}(A) } \overset{\eqref{Gamma square bound, another}}{\leq} \frac{L}{\lambda_{\min}} \sum_i L_i = \kappa \sum_i L_i.
$$
which proves the first part of \eqref{bound G'AinvG}.
\end{proof}

\subsection{Proof of Lemma for Equal Eigenvalues}\label{lm with proof: equal eigenvalues}



\begin{proof} 
Denote $\text{eig}(Z)$ as the set of eigenvalues of $Z$ (allow
repeated elements; e.g. if $Z$ has eigenvalues $1.7$ with multiplicity $2$, then 
we define $\text{eig}(Z) = \{ 1.7, 1.7 \} $). We first notice the following simple fact.

    \textbf{Fact}: For any square matrix $M$ with the following block structure:
    \begin{equation}\label{eq: upper trig matrix M}
        M = \begin{bmatrix}
            M_{11} & M_{12} \\
            \mathbf{0}_{n-1:n-1} & M_{22}
        \end{bmatrix},
    \end{equation}
    whrere $M_{11}, M_{12}$ and $M_{22}$ are submatrices of $M$ with appropriate shapes, $\mathbf{0}_{n-1:n-1}$ is an $n-1$ by $n-1$ zero matrix, 
    $  \text{eig}(M) = \text{eig}(M_{11}) \cup  \text{eig}(M_{22}) . $
    
The fact is simple to prove: the characteristic polynomial of $M$ is $\det(\lambda I - M) = \det( \lambda I -  M_{11} ) \det( \lambda I -  M_{22} )$. Since $ \det(\lambda I - Z)  = \Pi_{\mu \in \text{eig}(Z)} (\lambda - \mu ) $, we have $  \text{eig}(M) = \text{eig}(M_{11}) \cup  \text{eig}(M_{22}) . $

    Since $\Gamma$ is the lower triangular part of $Q$, then the first column of $Q^{-1} \Gamma$ is equal to the first column of the identity matrix, i.e. a column vector in the form: $(1, 0, \cdots, 0)^T$. 
    $Q^{-1} \Gamma$ can be represented in terms of its submatrices
    where $\mathbf{0}$ represents a column vector of all zeros in length of $n-1$ and write $B^{-1}$ and $\Gamma^T$ into block matrices:
    \begin{equation} \label{submatrix}
        Q^{-1} \Gamma = 
        \begin{bmatrix}
            1 & \mathbf{b_{12}} \\
            \mathbf{0} & \mathbf{b_{22}}
        \end{bmatrix}, \quad
        B^{-1} = \begin{bmatrix}
            1 & \mathbf{0}^T \\
            \mathbf{0} & \Gamma^T_{2:n}
        \end{bmatrix}, \quad
        \Gamma^T = \begin{bmatrix}
            1 & \mathbf{\Gamma^T_{1, 2:n}} \\
            \mathbf{0} & \Gamma^T_{2:n}
        \end{bmatrix},
    \end{equation}
    where $\mathbf{0}^T$ is a row vector of all zeros in length of $n-1$. $\mathbf{\Gamma^T_{1, 2:n}}$ is a row vector of length $n-1$ and its entries are the last $n-1$'s entries of the first row of $\Gamma^T$. $\Gamma^T_{2:n}$ is the submatrix of $\Gamma^T$ which does not contain the first row and the last row of $\Gamma^T$. Note that we assume $Q_{11}=1$, so the upper left block of $\Gamma^T$ is $1$. 

    Using \cref{submatrix}, we compute $B^{-1}Q^{-1} \Gamma$ and $\Gamma^T Q^{-1} \Gamma$ in the form \cref{eq: upper trig matrix M}.
        \begin{align}
        B^{-1}Q^{-1} \Gamma &  = 
        \begin{bmatrix}
            1 & \mathbf{0}^T \\
            \mathbf{0} & \Gamma^T_{2:n}
        \end{bmatrix}
        \begin{bmatrix}
            1 & \mathbf{b_{12}} \\
            \mathbf{0} & \mathbf{b_{22}}
        \end{bmatrix}
        = 
        \begin{bmatrix}
            1 & \mathbf{b_{12}} \\
            \mathbf{0} & \Gamma^T_{2:n} \mathbf{b_{22}}
        \end{bmatrix} \\
        \Gamma^T Q^{-1} \Gamma & = 
        \begin{bmatrix}
            1 & \mathbf{\Gamma^T_{1, 2:n}}  \\
            \mathbf{0} & \Gamma^T_{2:n}
        \end{bmatrix}
        \begin{bmatrix}
            1 & \mathbf{b_{12}} \\
            \mathbf{0} & \mathbf{b_{22}}
        \end{bmatrix}
        = 
        \begin{bmatrix}
            1 & \mathbf{b_{12}} + \mathbf{\Gamma^T_{1, 2:n}}   \mathbf{b_{22}}\\
            \mathbf{0} & \Gamma^T_{2:n} \mathbf{b_{22}}
        \end{bmatrix}
    \end{align}
    
     Therefore, the diagonal blocks of $B^{-1}Q^{-1}\Gamma$ and $\Gamma^T Q^{-1} \Gamma$ are the same, and thus, by the fact we introduced at the beginning of the proof, the eigenvalues of $B^{-1}Q^{-1} \Gamma$ and $\Gamma^T Q^{-1} \Gamma$ are the same. 
\end{proof}

\section{Intermediate Results in the Proof of the Lower Bounds}
\label{sec: sup proofs for GBS lower bound}

\subsection{Proof of \cref{lm: decompose Q inverse}}

\begin{proof}
Denote $\tilde{c} = 1-c$ and $J = ee^T$, where $e$ is a vector with $1$ on every entry. Observe that $Q = c J  + \tilde{c} I$ and
recall that by Sherman-Morrison formula \cite{Sherman1950}, if a matrix is in the form of $W + u v^T$, where $u, v$ are vectors and the matrix $W$ is nonsingular, then the inverse of the matrix $W +  u v^T $ is
\begin{align}
    (W +  u v^T)^{-1} = W^{-1} - \frac{W^{-1}uv^TW^{-1}}{1 - v^T W^{-1}u}.
\end{align}

Since $Q$ is invertible and $I = e e^T$, then we can apply the Sherman-Morrison formula to obtain the inverse of Q.

\begin{equation}
    Q^{-1} = (c J  + \tilde{c} I )^{-1}
    = \frac{1}{\tilde{c}} I - \frac{c}{\tilde{c} (\tilde{c}+ cn)} J.
\end{equation}
\end{proof}

\subsection{Proof of \cref{prop: a few props in GBS proof}}

\begin{proof}
\textbf{Part 1.} By assumption of $Q$ from \cref{eq:Ac def}, we know $Q$ is a positive definite matrix. With simple calculations, we know $Q$ has one eigenvalue $1-c$ with multiplicity $n-1$ and one
eigenvalue $1-c + cn $ with multiplicity $1$. Hence, $\|Q\| = 1-c + cn$.


\textbf{Part 2.}
We observe that
\begin{equation}\label{eq: rewrite Lte}
    \Gamma^T e   = 
    \begin{bmatrix}
    1 & c & \cdots & c & c\\
    0 & 1 & \cdots & c & c \\
    \vdots & \vdots & \ddots & \vdots& \vdots\\
    0 & 0 & \cdots & 1 &c \\
    0 & 0 & \cdots & 0 & 1\\
    \end{bmatrix}
    \begin{bmatrix}
       1 \\
       1 \\
        \vdots \\
       1 \\
       1
    \end{bmatrix}
     = 
    \begin{bmatrix}
        1 + c(n-1) \\
        1 + c(n-2)\\
        \vdots\\
        1 + c\\
        1
    \end{bmatrix}
\end{equation}
Based on the observation \cref{eq: rewrite Lte}, 
\begin{equation*}
\quad \|\Gamma^TJ\Gamma\|  = \|\Gamma^T e\|^2 = \sum_{i=1}^n (1 + c(n-i))^2 = n + cn(n-1) + \frac{c^2}{6} n(n-1)(2n-1).
\end{equation*}
%

\textbf{Part 3.}
\begin{proposition}[Proposition 3.2 in \cite{Sra2014}]\label{prop:suvrit}
    Let $M=\left[m_{ij}\right]=[\min (i, j)]$, define $\theta_{k} :=\frac{2 k \pi}{2 n+1}$. Then $\|M\| = \left(2+2 \cos \theta_{n}\right)^{-1} \gtrapprox 4 n^{2} / \pi^{2}$.
\end{proposition}
We observe that $\Gamma\Gamma^T - Q$  equals to  $c^2 M$ where each entry $M_{ij}$ is defined as $j - 1$ when $j \leq i$ and otherwise is $i - 1$.
Observe that the submatrix $M_{2:n}$ is a special matrix which satisfies the definition $M_{ij} = \min(i, j)$, so we can apply \cref{prop:suvrit}, which is proved in  \cite{Sra2014}, to approximate the spectral norm of $\Gamma\Gamma^T - Q$. Since the entries on the first row and the first column of $M$ are all zeros, then 
$$\|\Gamma\Gamma^T - Q\| = c^2 \|M\| = c^2 \|M_{2:n}\| \gtrapprox c^2  \frac{4 n^{2}}{\pi^{2}}.$$ 

\end{proof}

\section{Additional Experiments for ADMM Algorithms}\label{appendix: additional experiments}

\rvs{In \cref{sec:experiments ADMM}, we already illustrate that GBS-ADMM and sGS-ADMM have much worse performance than RP-ADMM in solving a specific problem \cref{eq: solve strongly convex} with a fairly large $c = 0.95$. Now we consider the problem with a smaller value of $c = 0.1$. As shown in \cref{tab: ratios for strongly convex ADMM true c = 0.1}, the the gap between GBS-ADMM and RP-ADMM is reduced to $\mathcal{O}(n)$, but the gap between sGS-ADMM and RP-ADMM remains as large as shown in \cref{tab:iterations for strongly convex ADMM for c equals 0.3}.}

\begin{table}[H]
  \centering
  \caption{Comparison of GBS-ADMM, sGS-ADMM, RP-ADMM for solving \cref{eq: solve strongly convex} where $c = 0.1$. The numbers in the left three columns represent the number of epochs to achieve relative error $1\mathrm{e}{-5}$. 
  The other columns represent the ratios of the epochs of GBS-ADMM and sGS-ADMM over the epochs of RP-ADMM.}
\begin{tabular}{|c|c|c|c|c|c|}
    \hline
         & \multicolumn{3}{c|}{Epochs} & \multicolumn{2}{c|}{Ratios} \\
    \hline
    n     & GBS-ADMM & sGS-ADMM  & RP-ADMM & GBS-ADMM & sGS-ADMM \\
    \hline
    100   & 98 & 103   & 21 & 4.9 & 5.1 \\
    \hline
    200   & 202 & 276   & 19 & 10.6 & 14.5 \\
    \hline
    400   & 382  & 757  & 18 & 21.2 & 42.0 \\
    \hline
    \end{tabular}%
    \label{tab: ratios for strongly convex ADMM true c = 0.1}
\end{table}

\rvs{One of the main contribution of this paper is to provide an example and show the slow convergence of GBS-ADMM and sGS-ADMM when compared to RP-ADMM in solving the example. 
To provide a complete comparison, we also add examples which show GBS-ADMM and sGS-ADMM could converge as fast as (or even better than) RP-ADMM. }
\rvs{In particular, we consider different problem instances such as randomly generated tridiagonal matrices and circulant Hankel matrices. 
Circulant Hankel matrix is generated with independent standard Gaussian entries. 
More specifically, we generate $\delta_1, \delta_2, \dots, \delta_{N} \sim \mathcal{N}(0,1) $ and let the entries of a circulant Hankel matrix denote as $A_{i,j} = \delta_{i+j-1}$ (define $\delta_{k} = \delta_{k-n}$ if $k>n$).
In the experiment, we choose the diagonal entry of our tridiagonal matrices to be one and generate each off-diagonal entry from a standard normal distribution.
}
\rvs{As shown in \cref{tab: iterations for strongly convex ADMM for other matrices}, the performance of GBS-ADMM is comparable to RP-ADMM, and sGS-ADMM are the fastest among the others. }

\begin{table}[H]
  \centering
  \caption{Comparison of GBS-ADMM, sGS-ADMM, RP-ADMM for solving \cref{eq: solve strongly convex} where $Q = A^T A$ where $A$ is circulant Hankel matrix or randomly generated tridiagonal matrix. The numbers in the last three columns represent the number of epochs to achieve relative error $1\mathrm{e}{-5}$.}
\begin{tabular}{|c|c|c|c|}
    \hline
         & \multicolumn{3}{c|}{Circulant Hankel Matrix} \\
    \hline
    n     & GBS-ADMM & sGS-ADMM  & RP-ADMM\\
    \hline
    25   & 28600 & 28473   & 27282 \\
    \hline
    50   & 136130 & 24816   & 279340 \\
    \hline
    100   & 715607  & 66105  & 752492 \\
    \hline
         & \multicolumn{3}{c|}{Tridiagonal Matrix} \\
    \hline
    n     & GBS-ADMM & sGS-ADMM  & RP-ADMM\\
    \hline
    25   & 8145 & 1630 &  5169 \\
    \hline
    50   & 6340 & 4267 &  35092 \\
    \hline
    100   & 43219 & 14645 & 26649 \\
    \hline
\end{tabular}%
\label{tab: iterations for strongly convex ADMM for other matrices}
\end{table}


\bibliographystyle{ieeetr}
\bibliography{main}

\begin{thebibliography}{10}

\bibitem{wright2015coordinate}
S.~J. Wright, ``Coordinate descent algorithms,'' {\em Mathematical
  Programming}, vol.~151, no.~1, pp.~3--34, 2015.

\bibitem{friedman2010regularization}
J.~Friedman, T.~Hastie, and R.~Tibshirani, ``Regularization paths for
  generalized linear models via coordinate descent,'' {\em Journal of
  statistical software}, vol.~33, no.~1, p.~1, 2010.

\bibitem{platt1999fast}
J.~Platt, ``Fast training of support vector machines using sequential minimal
  optimization. advances in kernel methods—support vector learning (pp.
  185--208),'' {\em AJ, MIT Press, Cambridge, MA}, 1999.

\bibitem{hsieh2008dual}
C.-J. Hsieh, K.-W. Chang, C.-J. Lin, S.~S. Keerthi, and S.~Sundararajan, ``A
  dual coordinate descent method for large-scale linear svm,'' in {\em
  Proceedings of the 25th international conference on Machine learning},
  pp.~408--415, 2008.

\bibitem{chang2011libsvm}
C.-C. Chang and C.-J. Lin, ``Libsvm: A library for support vector machines,''
  {\em ACM transactions on intelligent systems and technology (TIST)}, vol.~2,
  no.~3, pp.~1--27, 2011.

\bibitem{kolda2009tensor}
T.~G. Kolda and B.~W. Bader, ``Tensor decompositions and applications,'' {\em
  SIAM review}, vol.~51, no.~3, pp.~455--500, 2009.

\bibitem{shi2011iteratively}
Q.~Shi, M.~Razaviyayn, Z.-Q. Luo, and C.~He, ``An iteratively weighted mmse
  approach to distributed sum-utility maximization for a mimo interfering
  broadcast channel,'' {\em IEEE Transactions on Signal Processing}, vol.~59,
  no.~9, pp.~4331--4340, 2011.

\bibitem{shalev2013stochastic}
S.~Shalev-Shwartz and T.~Zhang, ``Stochastic dual coordinate ascent methods for
  regularized loss minimization,'' {\em Journal of Machine Learning Research},
  vol.~14, no.~Feb, pp.~567--599, 2013.

\bibitem{yen2015sparse}
I.~E.-H. Yen, K.~Zhong, C.-J. Hsieh, P.~K. Ravikumar, and I.~S. Dhillon,
  ``Sparse linear programming via primal and dual augmented coordinate
  descent,'' in {\em Advances in Neural Information Processing Systems},
  pp.~2368--2376, 2015.

\bibitem{chang2008coordinate}
K.-W. Chang, C.-J. Hsieh, and C.-J. Lin, ``Coordinate descent method for
  large-scale l2-loss linear support vector machines,'' {\em Journal of Machine
  Learning Research}, vol.~9, no.~Jul, pp.~1369--1398, 2008.

\bibitem{powell1973search}
M.~J. Powell, ``On search directions for minimization algorithms,'' {\em
  Mathematical programming}, vol.~4, no.~1, pp.~193--201, 1973.

\bibitem{gordon2015coordinate}
G.~Gordon and R.~Tibshirani, ``Coordinate descent,'' {\em Optimization},
  vol.~10, no.~36, p.~725, 2015.

\bibitem{sun2016worst}
R.~Sun and Y.~Ye, ``Worst-case complexity of cyclic coordinate descent: $ {O}
  (n^{2}) $ gap with randomized version,'' {\em arXiv preprint
  arXiv:1604.07130}, 2016.

\bibitem{leventhal2010randomized}
D.~Leventhal and A.~S. Lewis, ``Randomized methods for linear constraints:
  convergence rates and conditioning,'' {\em Mathematics of Operations
  Research}, vol.~35, no.~3, pp.~641--654, 2010.

\bibitem{nesterov2012efficiency}
Y.~Nesterov, ``Efficiency of coordinate descent methods on huge-scale
  optimization problems,'' {\em SIAM Journal on Optimization}, vol.~22, no.~2,
  pp.~341--362, 2012.

\bibitem{richtarik2014iteration}
P.~Richt{\'a}rik and M.~Tak{\'a}{\v{c}}, ``Iteration complexity of randomized
  block-coordinate descent methods for minimizing a composite function,'' {\em
  Mathematical Programming}, vol.~144, no.~1-2, pp.~1--38, 2014.

\bibitem{lu2015complexity}
Z.~Lu and L.~Xiao, ``On the complexity analysis of randomized block-coordinate
  descent methods,'' {\em Mathematical Programming}, vol.~152, no.~1-2,
  pp.~615--642, 2015.

\bibitem{zheng5873randomized}
Q.~Zheng, P.~Richtarik, and T.~Zhang, ``Randomized dual coordinate ascent with
  arbitrary sampling,'' {\em arXiv preprint arXiv:1411.5873}, 2015.

\bibitem{lin2015accelerated}
Q.~Lin, Z.~Lu, and L.~Xiao, ``An accelerated randomized proximal coordinate
  gradient method and its application to regularized empirical risk
  minimization,'' {\em SIAM Journal on Optimization}, vol.~25, no.~4,
  pp.~2244--2273, 2015.

\bibitem{zhang2017stochastic}
Y.~Zhang and L.~Xiao, ``Stochastic primal-dual coordinate method for
  regularized empirical risk minimization,'' {\em The Journal of Machine
  Learning Research}, vol.~18, no.~1, pp.~2939--2980, 2017.

\bibitem{fercoq2015accelerated}
O.~Fercoq and P.~Richt{\'a}rik, ``Accelerated, parallel, and proximal
  coordinate descent,'' {\em SIAM Journal on Optimization}, vol.~25, no.~4,
  pp.~1997--2023, 2015.

\bibitem{liu2015asynchronous}
J.~Liu, S.~J. Wright, C.~R{\'e}, V.~Bittorf, and S.~Sridhar, ``An asynchronous
  parallel stochastic coordinate descent algorithm,'' {\em The Journal of
  Machine Learning Research}, vol.~16, no.~1, pp.~285--322, 2015.

\bibitem{patrascu2015efficient}
A.~Patrascu and I.~Necoara, ``Efficient random coordinate descent algorithms
  for large-scale structured nonconvex optimization,'' {\em Journal of Global
  Optimization}, vol.~61, no.~1, pp.~19--46, 2015.

\bibitem{hsieh2015passcode}
C.-J. Hsieh, H.-F. Yu, and I.~S. Dhillon, ``Passcode: Parallel asynchronous
  stochastic dual co-ordinate descent.,'' in {\em ICML}, pp.~2370--2379, 2015.

\bibitem{coddington1997random}
P.~D. Coddington, ``Random number generators for parallel computers,'' {\em The
  NHSE Review}, 1997.

\bibitem{thomas2009comparison}
D.~B. Thomas, L.~Howes, and W.~Luk, ``A comparison of cpus, gpus, fpgas, and
  massively parallel processor arrays for random number generation,'' in {\em
  Proceedings of the ACM/SIGDA international symposium on Field programmable
  gate arrays}, pp.~63--72, 2009.

\bibitem{floudas2001encyclopedia}
C.~A. Floudas and P.~M. Pardalos, {\em Encyclopedia of optimization}, vol.~1.
\newblock Springer Science \& Business Media, 2001.

\bibitem{lee2018random}
C.-P. Lee and S.~J. Wright, ``Random permutations fix a worst case for cyclic
  coordinate descent,'' {\em IMA Journal of Numerical Analysis}, vol.~39,
  no.~3, pp.~1246--1275, 2018.

\bibitem{glowinski1975approximation}
R.~Glowinski and A.~Marroco, ``Approximation par \'{e}l\'{e}ments finis d'ordre
  un, et la r\'{e}solution, par p\'{e}nalisation-dualit\'{e} d'une classe de
  probl\`{e}mes de dirichlet non lin\'{e}aires,'' {\em ESAIM: Mathematical
  Modelling and Numerical Analysis-Mod{\'e}lisation Math{\'e}matique et Analyse
  Num{\'e}rique}, vol.~9, no.~R2, pp.~41--76, 1975.

\bibitem{chan1978finite}
T.~F.~C. Chan and R.~Glowinski, {\em Finite element approximation and iterative
  solution of a class of mildly non-linear elliptic equations}.
\newblock Computer Science Department, Stanford University Stanford, 1978.

\bibitem{gabay1976dual}
D.~Gabay and B.~Mercier, ``A dual algorithm for the solution of nonlinear
  variational problems via finite element approximation,'' {\em Computers \&
  Mathematics with Applications}, vol.~2, no.~1, pp.~17--40, 1976.

\bibitem{boyd2011distributed}
S.~Boyd, N.~Parikh, and E.~Chu, {\em Distributed optimization and statistical
  learning via the alternating direction method of multipliers}.
\newblock Now Publishers Inc, 2011.

\bibitem{tseng2001convergence}
P.~Tseng, ``Convergence of a block coordinate descent method for
  nondifferentiable minimization,'' {\em Journal of optimization theory and
  applications}, vol.~109, no.~3, pp.~475--494, 2001.

\bibitem{chen2016direct}
C.~Chen, B.~He, Y.~Ye, and X.~Yuan, ``The direct extension of admm for
  multi-block convex minimization problems is not necessarily convergent,''
  {\em Mathematical Programming}, vol.~155, no.~1-2, pp.~57--79, 2016.

\bibitem{hong2012linear}
M.~Hong and Z.-Q. Luo, ``On the linear convergence of the alternating direction
  method of multipliers,'' {\em arXiv preprint arXiv:1208.3922}, 2012.

\bibitem{han2012note}
D.~Han and X.~Yuan, ``A note on the alternating direction method of
  multipliers,'' {\em Journal of Optimization Theory and Applications},
  vol.~155, no.~1, pp.~227--238, 2012.

\bibitem{chen2013convergence}
C.~Chen, Y.~Shen, and Y.~You, ``On the convergence analysis of the alternating
  direction method of multipliers with three blocks,'' in {\em Abstract and
  Applied Analysis}, vol.~2013, Hindawi Publishing Corporation, 2013.

\bibitem{he2013proximal}
B.~He, H.-K. Xu, and X.~Yuan, ``On the proximal jacobian decomposition of alm
  for multiple-block separable convex minimization problems and its
  relationship to admm,'' 2013.

\bibitem{he2013full}
B.~He, L.~Hou, and X.~Yuan, ``On full jacobian decomposition of the augmented
  lagrangian method for separable convex programming,'' {\em SIAM Journal on
  Optimization}, vol.~25, no.~4, pp.~2274--2312, 2015.

\bibitem{lin2014convergence}
T.~Lin, S.~Ma, and S.~Zhang, ``On the convergence rate of multi-block admm,''
  {\em arXiv preprint arXiv:1408.4265}, vol.~229, 2014.

\bibitem{hong2014block}
M.~Hong, T.-H. Chang, X.~Wang, M.~Razaviyayn, S.~Ma, and Z.-Q. Luo, ``A block
  successive upper bound minimization method of multipliers for linearly
  constrained convex optimization,'' {\em arXiv preprint arXiv:1401.7079},
  2014.

\bibitem{cai2014direct}
X.~Cai, D.~Han, and X.~Yuan, ``The direct extension of {ADMM} for three-block
  separable convex minimization models is convergent when one function is
  strongly convex,'' {\em Optimization Online}, 2014.

\bibitem{sun2014convergent}
D.~Sun, K.-C. Toh, and L.~Yang, ``A convergent proximal alternating direction
  method of multipliers for conic programming with 4-block constraints,'' {\em
  arXiv preprint arXiv:1404.5378}, 2014.

\bibitem{lin2014global}
T.~Lin, S.~Ma, and S.~Zhang, ``On the global linear convergence of the admm
  with multi-block variables,'' {\em arXiv preprint arXiv:1408.4266}, 2014.

\bibitem{han2014augmented}
D.~Han, X.~Yuan, and W.~Zhang, ``An augmented lagrangian based parallel
  splitting method for separable convex minimization with applications to image
  processing,'' {\em Mathematics of Computation}, vol.~83, no.~289,
  pp.~2263--2291, 2014.

\bibitem{li2014schur}
X.~Li, D.~Sun, and K.-C. Toh, ``A schur complement based semi-proximal admm for
  convex quadratic conic programming and extensions,'' {\em Mathematical
  Programming}, pp.~1--41, 2014.

\bibitem{li2015convergent}
M.~Li, D.~Sun, and K.-C. Toh, ``A convergent 3-block semi-proximal admm for
  convex minimization problems with one strongly convex block,'' {\em
  Asia-Pacific Journal of Operational Research}, p.~1550024, 2015.

\bibitem{lin2015iteration}
T.~Lin, S.~Ma, and S.~Zhang, ``Iteration complexity analysis of multi-block
  admm for a family of convex minimization without strong convexity,'' {\em
  arXiv preprint arXiv:1504.03087}, 2015.

\bibitem{deng2017parallel}
W.~Deng, M.-J. Lai, Z.~Peng, and W.~Yin, ``Parallel multi-block admm with ${O}
  (1/k)$ convergence,'' {\em Journal of Scientific Computing}, vol.~71, no.~2,
  pp.~712--736, 2017.

\bibitem{he2012alternating}
B.~He, M.~Tao, and X.~Yuan, ``Alternating direction method with gaussian back
  substitution for separable convex programming,'' {\em SIAM Journal on
  Optimization}, vol.~22, no.~2, pp.~313--340, 2012.

\bibitem{he2012convergence}
B.~He, M.~Tao, and X.~Yuan, ``Convergence rate and iteration complexity on the
  alternating direction method of multipliers with a substitution procedure for
  separable convex programming,'' {\em Preprint 3611, Optimization Online},
  2012.

\bibitem{li2016schur}
X.~Li, D.~Sun, and K.-C. Toh, ``A schur complement based semi-proximal admm for
  convex quadratic conic programming and extensions,'' {\em Mathematical
  Programming}, vol.~155, no.~1-2, pp.~333--373, 2016.

\bibitem{chen2017efficient}
L.~Chen, D.~Sun, and K.-C. Toh, ``An efficient inexact symmetric gauss--seidel
  based majorized admm for high-dimensional convex composite conic
  programming,'' {\em Mathematical Programming}, vol.~161, no.~1-2,
  pp.~237--270, 2017.

\bibitem{sun2015expected}
R.~Sun, Z.-Q. Luo, and Y.~Ye, ``On the expected convergence of randomly
  permuted admm,'' {\em arXiv preprint arXiv:1503.06387}, 2015.

\bibitem{recht2012beneath}
B.~Recht and C.~R{\'e}, ``Beneath the valley of the noncommutative
  arithmetic-geometric mean inequality: conjectures, case-studies, and
  consequences,'' {\em arXiv preprint arXiv:1202.4184}, 2012.

\bibitem{wright2017analyzing}
S.~J. Wright and C.-P. Lee, ``Analyzing random permutations for cyclic
  coordinate descent,'' {\em arXiv preprint arXiv:1706.00908}, 2017.

\bibitem{gurbuzbalaban2015random}
M.~G{\"u}rb{\"u}zbalaban, A.~Ozdaglar, and P.~A. Parrilo, ``Why random
  reshuffling beats stochastic gradient descent,'' {\em Mathematical
  Programming}, pp.~1--36, 2019.

\bibitem{gurbuzbalaban2018randomness}
M.~G{\"u}rb{\"u}zbalaban, A.~Ozdaglar, N.~D. Vanli, and S.~J. Wright,
  ``Randomness and permutations in coordinate descent methods,'' {\em
  Mathematical Programming}, vol.~181, no.~2, pp.~349--376, 2020.

\bibitem{chen2015convergence}
C.~Chen, M.~Li, X.~Liu, and Y.~Ye, ``On the convergence of multi-block
  alternating direction method of multipliers and block coordinate descent
  method,'' {\em arXiv preprint arXiv:1508.00193}, 2015.

\bibitem{sun2015improved}
R.~Sun and M.~Hong, ``Improved iteration complexity bounds of cyclic block
  coordinate descent for convex problems,'' in {\em Advances in Neural
  Information Processing Systems}, pp.~1306--1314, 2015.

\bibitem{li2017faster}
X.~Li, T.~Zhao, R.~Arora, H.~Liu, and M.~Hong, ``On faster convergence of
  cyclic block coordinate descent-type methods for strongly convex
  minimization,'' {\em The Journal of Machine Learning Research}, vol.~18,
  no.~1, pp.~6741--6764, 2017.

\bibitem{angelos1992triangular}
J.~R. Angelos, C.~C. Cowen, and S.~K. Narayan, ``Triangular truncation and
  finding the norm of a hadamard multiplier,'' {\em Linear algebra and its
  applications}, vol.~170, pp.~117--135, 1992.

\bibitem{Sherman1950}
J.~Sherman and W.~J. Morrison, ``{Adjustment of an Inverse Matrix Corresponding
  to a Change in One Element of a Given Matrix},'' {\em The Annals of
  Mathematical Statistics}, vol.~21, pp.~124--127, mar 1950.

\bibitem{Sra2014}
S.~Sra, ``Explicit diagonalization of an anti-triangular cesar\'o matrix,''
  2014.

\end{thebibliography}

\end{document}